\documentclass[letterpaper,12pt]{article}
\usepackage{amsmath,amssymb,amsfonts,epsf,epsfig,enumerate,placeins}
\newtheorem{theo}{Theorem}[section]
\newtheorem{prop}[theo]{Proposition}
\newtheorem{lemma}[theo]{Lemma}

\newtheorem{cor}[theo]{Corollary}

\newtheorem{conj}[theo]{Conjecture}

\newtheorem{problem}[theo]{Problem}

\newenvironment{proof}{\noindent {\sc Proof}.}
                {\phantom{a} \hfill \framebox[2.2mm]{ } \bigskip}

\newtheorem{innercustomthm}{Theorem}
\newenvironment{customthm}[1]
  {\renewcommand\theinnercustomthm{#1}\innercustomthm}
  {\endinnercustomthm}

\newcommand{\ZZ}{\mathbb{Z}}

\def\int{{\rm int}}

\newcommand{\D}{{\mathcal{D}}}

\newcommand{\C}{{\mathcal{C}}}
\renewcommand{\P}{{\mathcal{P}}}
\newcommand{\F}{{\mathcal{F}}}
\renewcommand{\star}{\wr}

\title{On the directed Oberwolfach problem \\ for complete symmetric equipartite digraphs \\
and uniform-length cycles}
\author{Nevena Franceti\'{c} and Mateja \v{S}ajna\footnote{Email: msajna@uottawa.ca. Phone: +1-613-562-5800 ext. 3522. Mailing address: Department of Mathematics and Statistics, University of Ottawa, 150 Louis-Pasteur Private, Ottawa, ON, K1N 6N5,Canada.} \\ University of Ottawa}

\begin{document}
\maketitle \baselineskip 17pt

\begin{abstract}
We examine the necessary and sufficient conditions for a complete symmetric equipartite digraph $K_{n[m]}^\ast$ with $n$ parts of size $m$ to admit a resolvable decomposition into directed cycles of length $t$. We show that the obvious necessary conditions are sufficient for $m,n,t \ge 2$ in each of the following four cases: (i) $m(n-1)$ is even; (ii) $\gcd(m,n) \not\in \{1,3\}$; (iii) $\gcd(m,n)=1$ and $4|n$ or $6|n$; and (iv) $\gcd(m,n)=3$, and if $n=6$, then $p|m$ for a prime $p \le 37$.

\end{abstract}

\section{Introduction}

The celebrated Oberwolfach problem (OP), posed by Ringel in 1967, asks whether $n$ participants at a conference can be seated at $k$ round tables of sizes $t_1, t_2, \ldots, t_k$ for several nights in row so that each participant sits next to everybody else exactly once. The assumption is that $n$ is odd and $n=t_1+t_2+\ldots+t_k$. In graph-theoretic terms, OP$(t_1, t_2, \ldots, t_k)$ asks whether $K_n$ admits a decomposition into 2-factors, each a disjoint union of cycles of lengths $t_1, t_2, \ldots, t_k$. When $n$ is even, the complete graph minus a 1-factor, $K_n-I$, is considered instead \cite{HuaKot}. OP has been solved completely in the case that $t_1=t_2=\ldots=t_k$ \cite{AlsHag,AlsSch,HofSch}, and in many other special cases (for example, for $k=2$ \cite{Tra}; for $t_1, t_2, \ldots, t_k$ all even \cite{BryDan}; for $n \le 60$ \cite{AdaBry,Dez,FraHol,FraRos,SalDra}; and for $n$ sufficiently large \cite{GloJoo}), but is in general still open.

The Oberwolfach problem for the complete equipartite graph $K_{n[m]}$ with $n$ parts of size $m$ and uniform cycle lengths was completely solved by Liu, as stated below.

\begin{theo}{\rm \cite{Liu}}\label{thm:Liu}
Let $t \ge 3$ and $n \ge 2$. Then $K_{n[m]}$ admits a resolvable decomposition into cycles of length $t$ if and only if $t|mn$, $m(n-1)$ is even, $t$ is even when $n = 2$, and $(m, n, t) \not\in\{(2,3,3),(6,3,3),(2,6,3),(6,2,6)\}$.
\end{theo}

The directed Oberwolfach problem was introduced in \cite{BurSaj}. It asks whether $n$ participants can be seated at $k$ round tables of sizes $t_1, t_2, \ldots, t_k$ (where $n=t_1+t_2+\ldots+t_k$)
for several nights in row so that each person sits {\em to the right} of everybody else exactly once. Such a seating is equivalent to a decomposition of $K_n^*$, the complete symmetric digraph of order $n$, into subdigraphs isomorphic to a disjoint union of directed cycles of lengths $t_1, t_2, \ldots, t_k$. The solution to this problem for uniform cycle lengths has been completed very recently (see below), while very little is known about the non-uniform case.

\begin{theo}{\rm \cite{BerGerSot,BenZha,AdaBry1,BurSaj,BurFraSaj,Lac,Til}}\label{thm:Kn*}
Let $t \ge 2$ and $n \ge 2$. Then $K_n^*$ admits a resolvable decomposition into directed cycles of length $t$ if and only if $t|n$ and $(n, t) \not\in\{(6,3),(4,4),(6,6)\}$.
\end{theo}

In this paper, we introduce the directed Oberwolfach problem for complete symmetric equipartite digraphs. As a scheduling problem, it asks whether the $nm$ participants at a conference, consisting of $n$ delegations of $m$ participants each, can be seated at round tables of sizes $t_1, t_2, \ldots, t_k$ (where $nm=t_1+t_2+\ldots+t_k$) so that over the course of $m(n-1)$ meals, every participant sits {\em to the right of every participant from another delegation} exactly once. Thus, we are asking about the existence of a decomposition of $K_{n[m]}^*$, the complete symmetric equipartite digraph with $n$ parts of size $m$, into subdigraphs, each a disjoint union of directed cycles of lengths $t_1, t_2, \ldots, t_k$. Limiting our investigation to the uniform cycle length, we propose the following problem.

\begin{problem}\label{prob}{\rm
Determine the necessary and sufficient conditions on $m$, $n$, and $t$ for $K_{n[m]}^*$ to admit a resolvable decomposition into directed $t$-cycles.}
\end{problem}

Apart from case $m=1$ (Theorem~\ref{thm:Kn*}) and decompositions that follow directly from Theorem~\ref{thm:Liu} (see Corollary~\ref{cor:Liu} below), to our knowledge, the only previous contribution to Problem~\ref{prob} is a partial solution for $t=3$, as stated below.

\begin{theo}{\rm \cite{BenWeiZhu}}\label{thm:Ben}
The digraph $K_{n[m]}^*$ admits a resolvable decomposition into directed 3-cycles  if and only if $3|mn$ and $(m,n) \ne (1,6)$, with possible exceptions of the form $(m,6)$, where $m$ is not divisible by any prime less than 17.
\end{theo}

The main result of this paper is as follows.

\begin{theo}\label{thm:main}
Let $m$, $n$, and $t$ be integers greater than 1, and let $g=\gcd(n,t)$. Assume one of the following conditions holds.
\begin{enumerate}[(i)]
\item $m(n-1)$ even; or
\item $g \not\in \{ 1,3 \}$; or
\item $g=1$, and $n \equiv 0 \pmod{4}$ or $n \equiv 0 \pmod{6}$; or
\item $g=3$, and if $n=6$, then $m$ is divisible by a prime $p \le 37$.
\end{enumerate}
Then the digraph $K_{n[m]}^*$ admits a resolvable decomposition into directed $t$-cycles if and only if $t|mn$ and $t$ is even in case $n=2$.
\end{theo}

As we shall see, to complete Problem~\ref{prob}, it suffices to show that the obvious necessary conditions on $(m,n,t)$ are sufficient in the following two cases: (i) $(m,n,t)=(t,2p,t)$ for a prime $p \ge 5$  and odd prime $t$; and (ii) $(m,n,t)=(m,6,3)$ for a prime $m \ge 41$.

This paper is organized as follows. In Section 2 we introduce the necessary terminology, and in Section 3 we solve the easiest case of Problem~\ref{prob}, that is, the case with $m(n-1)$ even. In Section 4 we present some smaller decompositions that help us address the rest of the problem. In Section 5, we solve the easy cases for $m(n-1)$ odd, and address the difficult cases in Sections 6--9. The  proof of Theorem~\ref{thm:main}, as well as the outstanding cases of Problem~\ref{prob}, are summarized in Section 10.

\section{Prerequisites}

As usual, the vertex set and arc set of a directed graph (shortly {\em digraph}) $D$ will be denoted $V(D)$ and $A(D)$, respectively. All digraphs in this paper are strict, that is, have no loops and no parallel arcs.

By $K_n$, $\bar{K}_n$, $K_{m,n}$, $K_{n[m]}$, and $C_t$ we denote the complete graph of order $n$, the empty graph of order $n$, the complete bipartite graph with parts of size $m$ and $n$, the complete equipartite graph with $n$ parts of size $m$, and the  cycle of length $t$ ($t$-cycle), respectively.
Analogously, by $K_n^*$, $K_{m,n}^*$, $K_{n[m]}^*$, and $\vec{C}_t$ we denote the complete symmetric digraph of order $n$, the complete symmetric bipartite digraph with parts of size $m$ and $n$,
the complete symmetric equipartite digraph with $n$ parts of size $m$, and the directed cycle of length $t$ (directed $t$-cycle), respectively. A {\em $\vec{C}_t$-factor} of a digraph $D$ is a spanning subdigraph of $D$ that is a disjoint union of directed $t$-cycles.

A {\em decomposition} of a digraph $D$ is a set $\{ D_1, \ldots, D_k \}$ of digraphs of $D$ such that $\{ A(D_1), \ldots, A(D_k) \}$ is a partition of $A(D)$. A $D'$-decomposition of $D$, where $D'$ is a subdigraph of $D$, is a decomposition into subdigraphs isomorphic to $D'$.
A decomposition $\D=\{ D_1, \ldots, D_k \}$ of $D$ is said to be {\em resolvable} if $\D$ partitions into {\em parallel classes}, that is, sets $\{ D_{i_1}, \ldots, D_{i_{k_i}} \}$ such that $\{ V(D_{i_1}), \ldots, V(D_{i_{k_i}}) \}$ is a partition of $V(D)$.

A {\em $\vec{C}_t$-factorization} of $D$ is a decomposition of $D$ into $\vec{C}_t$-factors, and it corresponds to a resolvable $\vec{C}_t$-decomposition.

A decomposition, $C_t$-factor, and $C_t$-factorization of a graph are defined analogously.

The {\em wreath product} of digraphs $D_1$ and $D_2$, denoted $D_1 \wr D_2$, is the digraph with vertex set $V(D_1) \times V(D_2)$ and arc set $A(D_1 \wr D_2)$ consisting precisely of all arcs of the form $\left((u_1,u_2),(u_1,v_2)\right) $ where $(u_2, v_2) \in A(D_2)$, as well as all arcs of the form $\left( (u_1,u_2),(v_1,v_2) \right)$ where $(u_1,v_1) \in A(D_1)$.

It is not difficult to see that $K_n^\ast \wr K_m^* \cong K_{mn}^*$ and $K_n^\ast \wr \bar{K}_m \cong K_{n[m]}^*$.

\section{$\vec{C}_t$-factorization of $K_{n[m]}^*$: easy observations}

Throughout this paper we shall assume that $m$, $n$, and $t$ are integers greater than 1. The obvious necessary conditions for the existence of a $\vec{C}_{t}$-factorization of $K_n^\ast \star \bar{K}_m$ are as follows:
\begin{enumerate}[(C1)]
\item $t \vert mn$, and
\item $t$ is even when $n=2$.
\end{enumerate}

The following lemma, together with Theorem~\ref{thm:Liu}, will help us establish sufficiency in the case that $m(n-1)$ is even (Corollary~\ref{cor:Liu} below).

\begin{lemma}{\rm \cite{BurSaj,Ush}}\label{lem:bipartite}
Let $t \ge 2$ be an even integer, and $\beta$ any positive integer. Then the digraph $K_{\beta \frac{t}{2},\beta \frac{t}{2}}^\ast$ admits a $\vec{C}_{t}$-factorization.
\end{lemma}

\begin{cor}\label{cor:Liu}
Let $m(n-1)$ be even, let $t \ge 2$ be such that $t|mn$, and $t$ is even if $n=2$. Then $K_{n[m]}^\ast$ admits a $\vec{C}_{t}$-factorization.
\end{cor}

\begin{proof}
First, assume $t=2$. The graph $K_{n[m]}$ admits a $C_{mn}$-factorization by Theorem~\ref{thm:Liu}, and since $mn$ is even, it therefore admits a 1-factorization. Replacing each 1-factor in a 1-factorization of $K_{n[m]}$ with a $\vec{C}_{2}$-factor results in a $\vec{C}_{2}$-factorization of $K_{n[m]}^\ast$.

Hence we may now assume $t \ge 3$. If $(m,n,t) \not\in\{(2,3,3),(6,3,3),(2,6,3),(6,2,6)\}$, then by Theorem~\ref{thm:Liu}, since $m(n-1)$ is even, there exists a $C_t$-factorization of $K_{n[m]}$. To obtain a $\vec{C}_{t}$-factorization of $K_{n[m]}^*$, we direct each cycle in this decomposition in both possible ways.

Theorem~\ref{thm:Ben} guarantees existence of a $\vec{C}_{3}$-factorization of $K_{n[m]}^*$ for $(m,n) \in\{(2,3),(6,3),$ $(2,6)\}$.

Finally, let $(m,n,t)=(6,2,6)$, so $K_{n[m]}^* \cong K_{6,6}^\ast$. By Lemma~\ref{lem:bipartite}, there exists a $\vec{C}_6$-factorization of $K_{3,3}^\ast$. It is easy to see that $K_{6,6}^\ast$ admits a resolvable decomposition into copies of $K_{3,3}^\ast$. Hence $K_{6,6}^\ast$ admits a $\vec{C}_6$-factorization.
\end{proof}

\section{Some useful decompositions}

In this section, we prove existence of some $\vec{C}_t$-factorizations that will help us address Problem~\ref{prob} in the cases not covered by Corollary~\ref{cor:Liu}.

\begin{lemma}\label{lem:blow1}
Let $t \ge 3$ and $p$ be an odd prime. Then the following hold.
\begin{enumerate}[(1)]
\item There exists a $\vec{C}_t$-factorization of $\vec{C}_t \wr \bar{K}_p$.
\item There exists a $\vec{C}_{pt}$-factorization of $\vec{C}_t \wr \bar{K}_p$.
\item If $t$ is odd, then there exists a $\vec{C}_t$-factorization of $\vec{C}_t \wr \bar{K}_4$.
\end{enumerate}
\end{lemma}

\begin{proof}
For any $s \in \ZZ^+$, let the vertex set and arc set of $\vec{C}_t \wr \bar{K}_s$  be
$$V=\{ x_{j,i}: j \in \ZZ_t, i \in \ZZ_s \} \quad \mbox{ and } \quad A=\{ (x_{j,i_1},x_{j+1,i_2}): j \in \ZZ_t, i_1,i_2 \in \ZZ_s \},$$
respectively. We shall call an arc of the form $(x_{j,i},x_{j+1,i+d})$, for $d \in \ZZ_s$, an arc of {\em $j$-difference} $d$.
Moreover, define a permutation $\rho$ on $V$ by
$$\rho=( x_{0,0} \, x_{0,1} \, \ldots x_{0,s-1} )
( x_{1,0} \, x_{1,1} \, \ldots x_{1,s-1} ) \ldots
( x_{t-1,0} \, x_{t-1,1} \, \ldots x_{t-1,s-1} ).$$

For Claims (1) and (2), we have $s=p$, an odd prime, and we let $\delta=0$ for Claim (1), and $\delta=1$ for Claim (2). In both cases, as we show below, it suffices to find elements $d_j^{(i)} \in \ZZ_p$, for $j \in \ZZ_t$ and $i \in \ZZ_p$, such that
$$\sum_{j=0}^{t-1} d_j^{(i)}=\delta \qquad \mbox{for all } i \in \ZZ_p, \eqno{(*)}$$
and $$d_j^{(0)},d_j^{(1)},\ldots, d_j^{(p-1)} \qquad \mbox{are pairwise distinct for each } j \in \ZZ_t.$$

If $t-1 \not\equiv 0 \pmod{p}$, then we may choose
$$d_0^{(i)}=\ldots=d_{t-2}^{(i)}=i \qquad \mbox{ and }  \qquad d_{t-1}^{(i)}=\delta-(t-1)i.$$
Otherwise, that is, if $t-1 \equiv 0 \pmod{p}$, then $t-2 \not\equiv 0 \pmod{p}$, and we choose
$$d_0^{(i)}=\ldots=d_{t-3}^{(i)}=i \qquad \mbox{ and } \qquad d_{t-2}^{(i)}=d_{t-1}^{(i)}=2^{-1}(\delta-(t-2)i).$$

\begin{enumerate}[(1)]
\item Let $\delta=0$ and suppose we have $d_j^{(i)} \in \ZZ_p$, for $j \in \ZZ_t$ and $i \in \ZZ_p$, satisfying Condition $(*)$. Fix $i \in \ZZ_p$ and define the following directed closed walk in $\vec{C}_t \wr \bar{K}_p$:
    $$C^{(i)}=(x_{0,0}, x_{1,d_0^{(i)}}, x_{2,d_0^{(i)}+d_1^{(i)}}, \ldots, x_{t-1,\sum_{j=0}^{t-2} d_j^{(i)}}, x_{0,0} ).$$
    It is easy to see that $C^{(i)}$ is in fact a directed $t$-cycle. Since $\sum_{j=0}^{t-1} d_j^{(i)}=0$, it contains exactly one arc of each $j$-difference $d_j^{(i)}$, for $j \in \ZZ_t$.

    Let $F^{(i)}= C^{(i)} \cup \rho(C^{(i)}) \cup \ldots \cup \rho^{p-1}(C^{(i)})$, and it can be verified that  $F^{(i)}$
     is a $\vec{C}_t$-factor of $\vec{C}_t \wr \bar{K}_p$. Moreover, the directed cycles in $F^{(i)}$ jointly contain all arcs of $j$-difference $d_j^{(i)}$, for all $j \in \ZZ_t$.

    Since for all $j \in \ZZ_t$, we have that $d_j^{(0)},d_j^{(1)}, \ldots, d_j^{(p-1)}$ are pairwise distinct, it follows that
    $\F= \{ F^{(i)}:  i \in \ZZ_p \}$
    is a $\vec{C}_t$-factorization of $\vec{C}_t \wr \bar{K}_p$.

\item Now let $\delta=1$ and suppose we have $d_j^{(i)} \in \ZZ_p$, for $j \in \ZZ_t$ and $i \in \ZZ_p$, satisfying Condition $(*)$. Fix $i \in \ZZ_p$ and define the following directed closed walk in $\vec{C}_t \wr \bar{K}_p$:
    \begin{eqnarray*}
    C^{(i)} &=& (x_{0,0}, x_{1,d_0^{(i)}}, x_{2,d_0^{(i)}+d_1^{(i)}}, \ldots, x_{t-1,\sum_{j=0}^{t-2} d_j^{(i)}}, \\
    &&
    x_{0,1}, x_{1,1+d_0^{(i)}}, x_{2,1+d_0^{(i)}+d_1^{(i)}}, \ldots, x_{t-1,1+\sum_{j=0}^{t-2} d_j^{(i)}}, \\
    && \ldots, \\
    &&
    x_{0,p-1}, x_{1,p-1+d_0^{(i)}},
    x_{2,p-1+d_0^{(i)}+d_1^{(i)}}, \ldots, x_{t-1,p-1+\sum_{j=0}^{t-2} d_j^{(i)}},
    x_{0,0} ).
    \end{eqnarray*}
    Since $\sum_{j=0}^{t-1} d_j^{(i)}=1$, we have that $C^{(i)}$ is  a directed $pt$-cycle, and
    it contains all arcs of each $j$-difference $d_j^{(i)}$, for $j \in \ZZ_t$.

    Since for all $j \in \ZZ_t$, we have that $d_j^{(0)},d_j^{(1)}, \ldots, d_j^{(p-1)}$ are pairwise distinct, it follows that
    $$\C= \{ C^{(i)}:  i \in \ZZ_t \}$$
    is a $\vec{C}_{pt}$-decomposition and hence a $\vec{C}_{pt}$-factorization of $\vec{C}_t \wr \bar{K}_p$.

\begin{figure}[t]
\centerline{\includegraphics[scale=0.6]{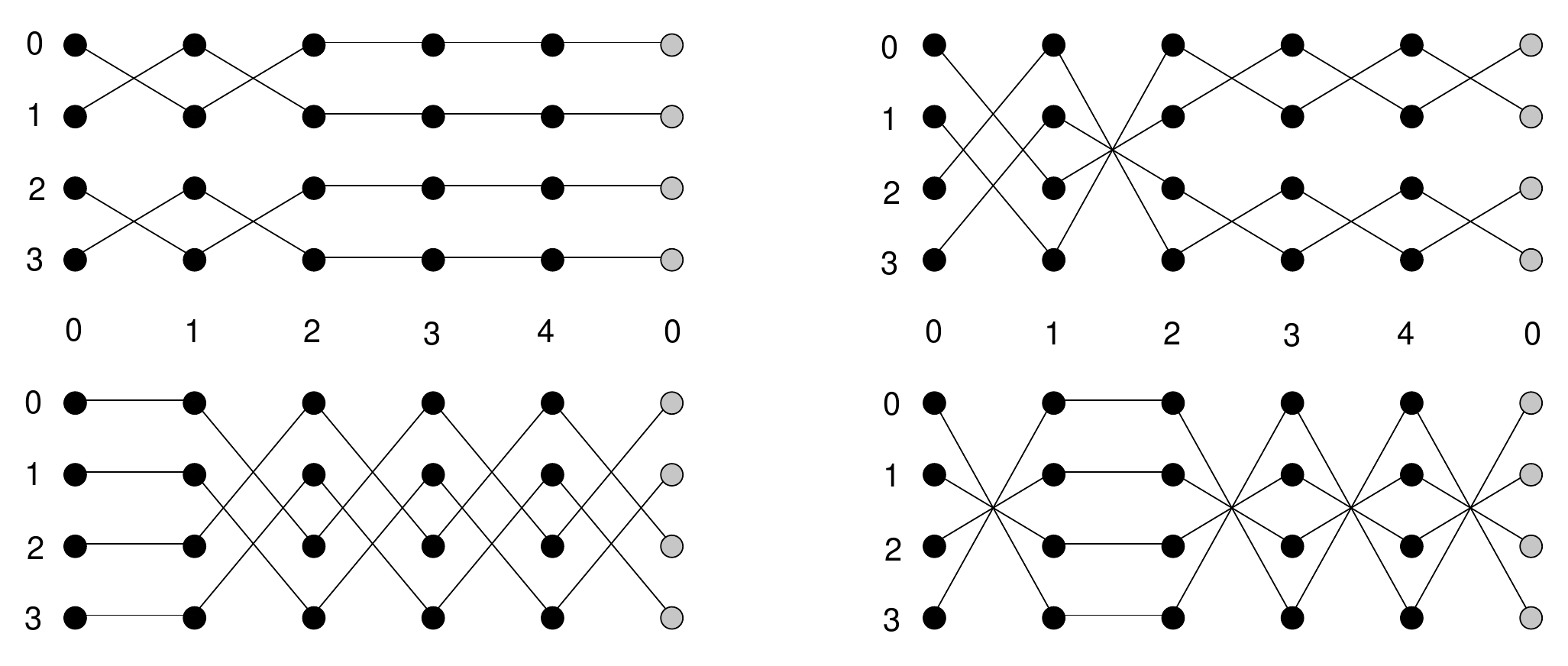}}
\caption{$\vec{C}_t$-factors $F^{(0)}, F^{(1)}$ (top), and $F^{(2)}, F^{(3)}$ (bottom) in a $\vec{C}_t$-factorization of $\vec{C}_t \wr \bar{K}_4$ for $t=5$. (All arcs are oriented from left to right, and only the subscripts of the vertices are specified.)}
\label{fig:pic1}
\end{figure}

\item We now have $s=4$, and we define another permutation, $\tau$, on $V$ by
$$\tau=( x_{0,0} \, x_{0,1} ) ( x_{0,2} \, x_{0,3} )
( x_{1,0} \, x_{1,1} ) ( x_{1,2} \, x_{1,3} ) \ldots
( x_{t-1,0} \, x_{t-1,1} ) ( x_{t-1,2} \, x_{t-1,3} ).$$
Define the following directed $t$-cycles in $\vec{C}_t \wr \bar{K}_4$.
\begin{eqnarray*}
C_0^{(0)} &=& (x_{0,0}, x_{1,1}, x_{2,0},x_{3,0}, x_{4,0}, x_{5,0}, \ldots,x_{t-1,0},x_{0,0} ) \\
C_0^{(1)} &=& (x_{0,0}, x_{1,2}, x_{2,1},x_{3,0}, x_{4,1}, x_{5,0}, \ldots,x_{t-1,1},x_{0,0} ) \\
C_0^{(2)} &=& (x_{0,0}, x_{1,0}, x_{2,2},x_{3,0}, x_{4,2}, x_{5,0}, \ldots,x_{t-1,2},x_{0,0} ) \\
C_0^{(3)} &=& (x_{0,0}, x_{1,3}, x_{2,3},x_{3,0}, x_{4,3}, x_{5,0}, \ldots,x_{t-1,3},x_{0,0} )
\end{eqnarray*}
Then, for each $i \in \ZZ_4$, let
$$C_1^{(i)}=\tau(C_0^{(i)}), \quad C_2^{(i)}=\rho^2(C_0^{(i)}), \quad \mbox{and} \quad C_3^{(i)}=\tau(C_2^{(i)}),$$
and let $F^{(i)}=C_0^{(i)} \cup C_1^{(i)} \cup C_2^{(i)} \cup C_3^{(i)}$.
Figure~\ref{fig:pic1} illustrates the case $t=5$. It is not difficult to verity that each $F^{(i)}$ is a $\vec{C}_t$-factor in $\vec{C}_t \wr \bar{K}_4$, and that $F^{(0)},\ldots,F^{(3)}$, for each $j \in \ZZ_t$, jointly contain exactly one arc of each $j$-difference. Hence
$\F= \{ F^{(i)}:  i \in \ZZ_4 \}$
    is a $\vec{C}_t$-factorization of $\vec{C}_t \wr \bar{K}_4$.
\end{enumerate}
\end{proof}

\begin{cor}\label{cor:blow-digraph}
Let $t\ge 3$ be an integer,  and let $D$ be a digraph admitting a $\vec{C}_t$-factorizaton. Let $s \ge 3$ be an odd integer, and $\ell$ a non-negative integer. Then the following hold.
\begin{enumerate}[(a)]
\item The digraph $D \wr \bar{K}_{s}$ admits a $\vec{C}_t$-factorizaton.
\item The digraph $D \wr \bar{K}_{s}$ admits a $\vec{C}_{st}$-factorizaton.
\item If $t$ is odd, then the digraph $D \wr \bar{K}_{4^{\ell}s}$ admits a $\vec{C}_t$-factorizaton.
\end{enumerate}
\end{cor}

\begin{proof}
\begin{enumerate}[(a)]
\item Let  $\cal C$ be a $\vec{C}_t$-factorizaton of $D$, and take any odd prime $p|s$.
    Then $\{ F \wr \bar{K}_{p}: F \in \cal C \}$ is a decomposition of $D \wr \bar{K}_{p}$ into spanning subdigraphs whose connected components are isomorphic to $\vec{C}_t \wr \bar{K}_{p}$. By Lemma~\ref{lem:blow1}(1), each such component admits a $\vec{C}_t$-factorizaton. Therefore, $D \wr \bar{K}_{p}$ admits a $\vec{C}_t$-factorizaton.

    Since for primes $p$ and $p'$ we have that $(D \wr \bar{K}_{p}) \wr \bar{K}_{p'} \cong D \wr \bar{K}_{pp'}$, repeating this process for all prime divisors of $s$ yields the desired result.

\item This is similar to (a), using Lemma~\ref{lem:blow1}(2).

\item This is similar to (a), using Lemma~\ref{lem:blow1}(1) and (3).
\end{enumerate}
\end{proof}

The above corollary shows how to ``blow up the holes'' in a $\vec{C}_t$-factorizaton by either keeping the cycle length, or ``blowing up'' the cycle length by the same odd factor. Note that Statement (b) also follows from \cite[Lemma 2.11]{PauSiv}, and Statement (a) can be obtained from \cite[Corollary 5.7]{Pio} by appropriately orienting each cycle.

\section{$\vec{C}_t$-factorizaton of $K_{n[m]}^*$ for $m$ odd, $n$ even: the easy cases}

%In the rest of the paper, we set $g=\gcd(n,t)$.

\begin{prop}\label{pro:easy}
Let $m$, $n$, and $t$ be integers greater than 1 with $m(n-1)$ odd, $t|mn$, and $t$ even if $n=2$. Furthermore, let $g=\gcd(n,t)$. Then $K_{n[m]}^\ast$ admits a $\vec{C}_t$-factorizaton in each of the following cases:
\begin{enumerate}[(1)]
\item $g$ is even and $(g,n) \not\in \{ (4,4),(6,6) \}$; and
\item $g$ is odd, $g \ge 3$, and $(g,n) \ne (3,6)$.
\end{enumerate}
\end{prop}

\begin{proof} Recall that $K_{n[m]}^\ast  \cong  K_n^\ast \star \bar{K}_{m}$. From the assumptions on $m$, $n$, and $t$ it follows that $m$ is odd, $n$ is even, $\frac{t}{g}$ is odd and divides $m$, and $\frac{mg}{t}$ is odd as well.
\begin{enumerate}[(1)]
\item Let $g$ be even. Assume first that $g \ge 4$. Since $g|n$  and $(g,n) \not\in \{ (4,4),(6,6)\}$, by Theorem~\ref{thm:Kn*}, there exists a $\vec{C}_g$-factorizaton of $K_n^\ast$. Hence, by Corollary~\ref{cor:blow-digraph}(b),  there exists a $\vec{C}_t$-factorizaton of $K_n^\ast \star \bar{K}_{\frac{t}{g}}$. Finally, by Corollary~\ref{cor:blow-digraph}(a),  there exists a $\vec{C}_t$-factorizaton of $K_n^\ast \star \bar{K}_{m}$.

Now let $g=2$, which implies $\frac{t}{2}|m$.
Since $n$ is even, $K_n$ admits a 1-factorization. Consequently, $K_n^\ast \star \bar{K}_m$ admits a resolvable decomposition into copies of $K_{m,m}^\ast$. Since $\frac{t}{2} \vert m$, by Lemma~\ref{lem:bipartite}, there exists a $\vec{C}_t$-factorizaton of $K_{m,m}^\ast$. Therefore, $K_n^\ast \star \bar{K}_m$ admits a $\vec{C}_t$-factorizaton.

\item Let $g$ be odd, $g \ge 3$.

First, assume $g=3$ and $n \ne 6$. By Theorem~\ref{thm:Ben}, there exists a $\vec{C}_{3}$-factorizaton of $K_n^\ast \wr \bar{K}_{\frac{3m}{t}}$. Hence by Corollary~\ref{cor:blow-digraph}(b), there exists a $\vec{C}_t$-factorizaton of $K_n^\ast \wr \bar{K}_{m}$.

Finally, let $g \ge 5$. Since $g|n$, by Theorem~\ref{thm:Kn*}, there exists a $\vec{C}_g$-factorizaton of $K_{n}^\ast$.  Hence by Corollary~\ref{cor:blow-digraph}(b), there exists a $\vec{C}_t$-factorizaton of $K_{n}^\ast  \wr \bar{K}_{\frac{t}{g}}$, and thus by Corollary~\ref{cor:blow-digraph}(a), there exists a $\vec{C}_t$-factorizaton of $K_{n}^\ast  \wr \bar{K}_{m}$.
\end{enumerate}
\end{proof}

Note that Proposition~\ref{pro:easy} leaves open only the following cases of Problem~\ref{prob} with  $m$ odd, $n$ even,  and $g=\gcd(n,t)$: case
$g=1$ and cases $(g,n)\in \{(3,6),(4,4),(6,6)\}$.

\section{$\vec{C}_t$-factorizaton of $K_{n[m]}^*$ for $m$ odd, $n$ even: the case $\gcd(n,t)=1$}

%Note that the above assumptions imply that $t|m$, $t$ is odd, and $n \ge 4$.

The following lemma and its corollary will allow us to reduce this case to a few crucial subcases, namely, to subcases $n=4$, $n=8$, and $n=2p$ for an odd prime $p$.

\begin{lemma}\label{lem:filling}
Let $t \ge 3$ be odd, $n_1 \ge 3$, and $n_2=4^\ell s$ for some integer $\ell \ge 0$ and odd integer $s \ge 1$. Assume that both $K_{n_1[t]}^\ast$ and $K_{n_2[t]}^\ast$ admit  $\vec{C}_t$-factorizations. Then $K_{n_1n_2[t]}^\ast$ admits a $\vec{C}_t$-factorization.
\end{lemma}

\begin{proof}
As $K_{n_1[t]}^\ast$ admits a $\vec{C}_t$-factorization, by Corollary~\ref{cor:blow-digraph}(c), so does $K_{n_1[t]}^\ast \wr \bar{K}_{4^\ell s} \cong K_{n_1[4^\ell st]}^\ast$.
Since $K_{n_1n_2[t]}^\ast \cong K_{n_1}^\ast \wr K_{n_2[t]}^\ast$ decomposes into $K_{n_1[4^\ell st]}^\ast$ and $n_1$ pairwise disjoint copies of $K_{n_2[t]}^\ast$, which by assumption admits a $\vec{C}_t$-factorization, we conclude that $K_{n_1n_2[t]}^\ast$ admits a $\vec{C}_t$-factorization.
\end{proof}

\begin{cor}\label{cor:reduction}
Let $t$ be odd, $t \ge 3$.
\begin{enumerate}[(1)]
\item Assume that each of $K_{4[t]}^\ast$ and $K_{8[t]}^\ast$ admits a $\vec{C}_t$-factorization. Then there exists a $\vec{C}_t$-factorization of $K_{n[t]}^\ast$ for all $n \equiv 0 \pmod{4}$.
\item Let $p$ be an odd prime, and assume that $K_{2p[t]}^\ast$ admits a $\vec{C}_t$-factorization. Then there exists a $\vec{C}_t$-factorization of $K_{n[t]}^\ast$ for all $n=2ps$ with $s$ odd.
\item Assume there exists a $\vec{C}_t$-factorization of $K_{n[t]}^\ast$ for all $n \in \{4,8\} \cup \{ 2p: p \mbox{ an odd prime} \}$. Then there exists a $\vec{C}_t$-factorization of $K_{n[t]}^\ast$ for all even $n \ge 4$.
\end{enumerate}
\end{cor}

\begin{proof}
\begin{enumerate}[(1)]
\item Take any $n \equiv 0 \pmod{4}$. There are two cases to consider.

    Case 1:  $n=4^\ell s$ with $\ell \ge 1$ and $s$ odd. If $s=1$, then a repeated application of Lemma~\ref{lem:filling} with $n_1=4$ and $n_2=4, 4^2, \ldots, 4^{\ell-1}$ yields a $\vec{C}_t$-factorization of $K_{n[t]}^\ast$. If $s \ge 3$, then by Corollary~\ref{cor:Liu}, there exists a $\vec{C}_t$-factorization of $K_{s[t]}^\ast$. We can now use Lemma~\ref{lem:filling} with $n_1=4$ and $n_2=s,4s, 4^2s, \ldots, 4^{\ell-1}s$.

Case 2: $n=8 \cdot 4^\ell s$ with $\ell \ge 0$ and $s$ odd, and we may assume that $\ell \ge 1$ or $s \ge 3$. Hence, by Corollary~\ref{cor:Liu} and Case 1, there exists a $\vec{C}_t$-factorization of $K_{4^\ell s[t]}^\ast$. We can therefore use
 Lemma~\ref{lem:filling} with $n_1=8$ and $n_2=4^\ell s$.

\item Use Lemma~\ref{lem:filling} with $n_1=2p$ and $n_2=s$.
\item This follows directly from (1) and(2).
\end{enumerate}
\end{proof}

\subsection{Subcase $n \equiv 0 \pmod{4}$}

In the next two lemmas, we show that the assumptions from Corollary~\ref{cor:reduction}(1) indeed hold, that is, both $K_{4[t]}^\ast$  and $K_{8[t]}^\ast$ admit $\vec{C}_t$-factorizations.

\begin{lemma}\label{lem:4}
Let $t$ be odd, $t \ge 3$. Then $K_{4[t]}^\ast$ admits a $\vec{C}_t$-factorization.
\end{lemma}

\begin{proof}
A $\vec{C}_t$-factorization of $K_{4[3]}^\ast $ exists by Theorem~\ref{thm:Ben}. Hence we may assume $t \ge 5$. We shall construct a $\vec{C}_t$-factorization of $K_{4[t]}^\ast$ as follows.

Let the vertex set of $D=K_{4[t]}^\ast $ be $V \cup X$, where $V$ and $X$ are disjoint sets, with $V=\{ v_i: i \in \ZZ_{3t}\}$ and $X=\{ x_i: i \in \ZZ_t\}$. The four parts (holes) of $D$ are $X$ and $V_r=\{ v_{3i+r}: i=0,1,\ldots,t-1 \}$, for $r=0,1,2$.
Note that $D[V]$ is a circulant digraph with connection set (set of differences) $\D=\{ d \in \ZZ_{3t}: d \not\equiv 0 \pmod{3}\}$. Define the permutation $\rho=( v_0 \, v_1 \, \ldots \, v_{3t-1})$ on $V \cup X$, which fixes the  vertices of $X$ pointwise.

Let $t=2k+1$. Hence the differences in $\D$ and the subscripts of the vertices in $V$ can be seen as elements of $\{ 0, \pm 1, \pm 2, \ldots, \pm (3k+1) \}$.

\begin{figure}[t]
\centerline{\includegraphics[scale=0.6]{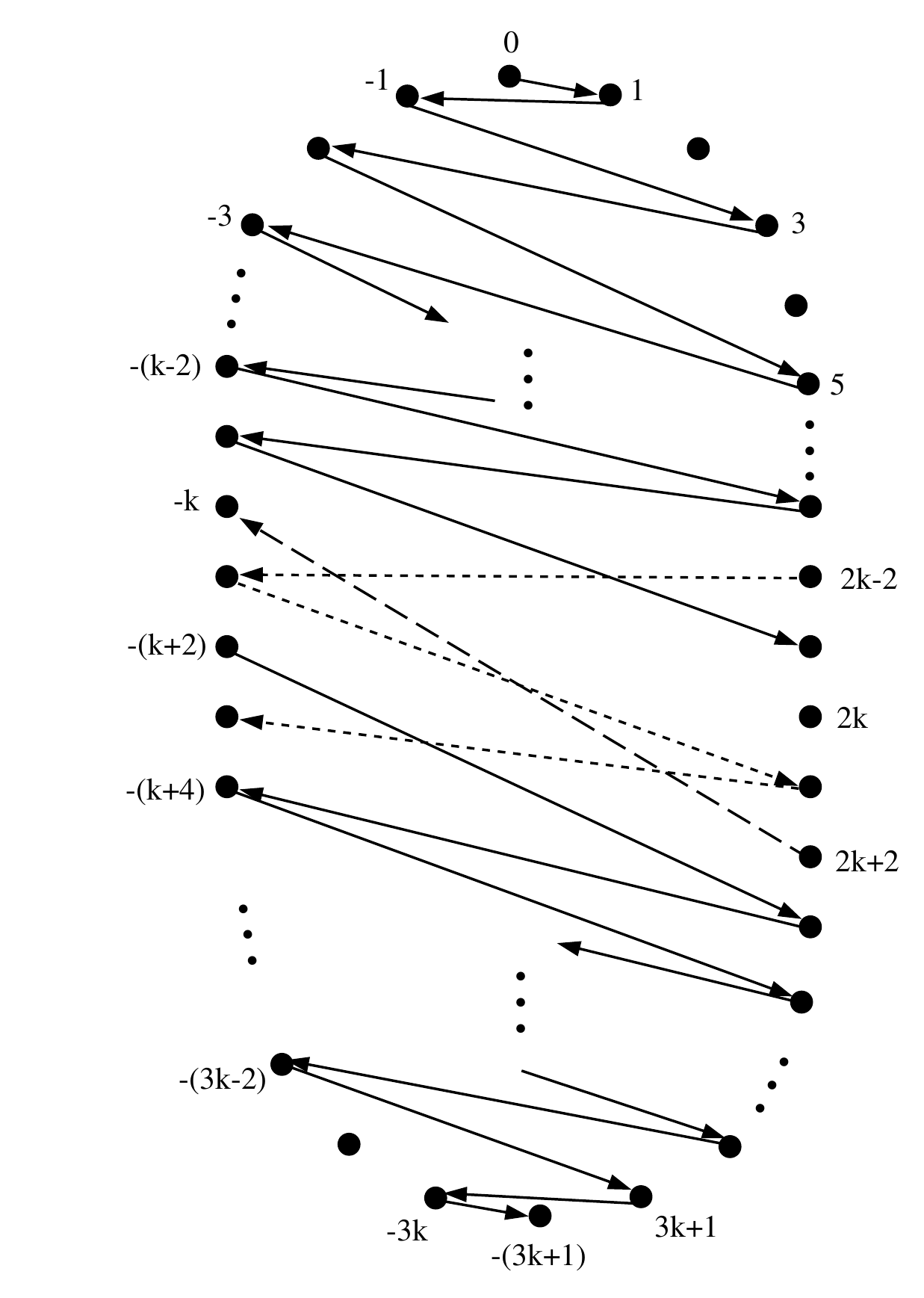}}
\caption{Directed paths $P_1,\ldots,P_4$ in the construction of a $\vec{C}_t$-factorization of $K_{4[t]}^\ast$. (All the vertices are in $V$, and only their subscripts are specified.) }
\label{fig:n=4}
\end{figure}

We define the following directed paths in $D[V]$ (see Figure~\ref{fig:n=4}):
$$P_1 = v_0 v_1 v_{-1} v_3 v_{-2} v_5 \ldots v_{2k-3} v_{-(k-1)} v_{2k-1},$$
and $P_2$ is obtained from $P_1$ by applying $\rho^{3k+2}$ (or $\rho^{-(3k+1)}$) and reversing the direction of the path. That is,
$$P_2= v_{-(k+2)} v_{2k+3} v_{-(k+4)} \ldots v_{-(3k-2)} v_{3k+1} v_{-3k} v_{-(3k+1)}.$$
Observe that $P_1$ and $P_2$ are disjoint, and jointly contain all vertices in $V$ except those in
\begin{eqnarray*}
V-(V(P_1) \cup V(P_2)) &=& \{ v_2,v_4,\ldots,v_{2k-2} \} \cup \{ v_{2k}, v_{2k+1},v_{2k+2} \} \\
&& \cup \{ v_{-(3k-1)},v_{-(3k-3)},\ldots,v_{-(k+3)} \} \cup \{ v_{-(k+1)},v_{-k} \}.
\end{eqnarray*}
The set of differences of the arcs in $P_1$, listing the differences in order of appearance, is
$$\D(P_1) = \{ 1,-2,4,-5,7,\ldots,3k-5,-(3k-4),3k-2 \},$$
and $\D(P_2) = - \D(P_1)$.

Furthermore, let
\begin{eqnarray*}
P_3 &=& v_{2k-2} v_{-(k+1)} v_{2k+1} v_{-(k+3)} \quad \mbox{ and} \\
P_4 &=& v_{2k+2} v_{-k}.
\end{eqnarray*}
Thus
\begin{eqnarray*}
\D(P_3) &=& \{  -(3k-1), -(3k+1),3k-1\} \quad \mbox{ and} \\
\D(P_4) &=& \{ 3k+1 \}.
\end{eqnarray*}
Observe that directed paths $P_1,\ldots,P_4$ are pairwise disjoint, and jointly contain exactly one arc of each difference in $\D$.

Let $U=V- \bigcup_{i=1}^4 V(P_i)$. It is easy to verify that $|U|=(6k+3)-(2k+2k+4+2)=2k-3$, so we may set $U=\{ u_0,\ldots,u_{2k-4} \}$.
Finally, we extend the four paths to four pairwise disjoint directed $t$-cycles as follows:
\begin{eqnarray*}
C_1 &=& P_1 v_{2k-1} x_0 v_0, \\
C_2 &=& P_2 v_{-(3k+1)} x_1  v_{-(k+2)}, \\
C_3 &=& P_3 v_{-(k+3)} x_2 u_0 x_3 u_1 \ldots u_{k-3} x_{k} v_{2k-2},  \quad \mbox{ and}\\
C_4 &=& P_4 v_{-k} x_{k+1} u_{k-2} x_{k+2} u_{k-1} \ldots  u_{2k-4} x_{2k} v_{2k+2} .
\end{eqnarray*}

Let $R= C_1 \cup C_2 \cup C_3 \cup C_4$, so $R$ is a $\vec{C}_t$-factor in $D$. It is not difficult to verify that
$\{ \rho^i(R): i \in \ZZ_{3t} \}$ is a  $\vec{C}_t$-factorization of $D$.
\end{proof}

\FloatBarrier

%%%% 8 holes

\begin{lemma}\label{lem:8}
Let $t$ be odd, $t \ge 3$. Then $K_{8[t]}^\ast$ admits a $\vec{C}_t$-factorization.
\end{lemma}

\begin{proof}
By Corollary~\ref{cor:blow-digraph}(b), we may assume that $t$ is a prime, and hence $t \equiv 1$ or $5 \pmod{6}$, and by Theorem~\ref{thm:Ben}, we may assume $t \ge 5$.

Let the vertex set of $D=K_{8[t]}^\ast$ be $V \cup X$, where $V$ and $X$ are disjoint sets, with $V=\{ v_i: i \in \ZZ_{7t}\}$ and $X=\{ x_i: i \in \ZZ_t\}$. The eight parts (holes) of $D$ are $X$ and $V_r=\{ v_{7i+r}: i=0,1,\ldots,t-1 \}$, for $r=0,1,\ldots,6$.
Note that $D[V]$ is a circulant digraph with connection set (set of differences) $\D=\{ d \in \ZZ_{7t}: d \not\equiv 0 \pmod{7}\}$.
Define the permutation $\rho=( v_0 \, v_1 \, \ldots \, v_{7t-1})$, which fixes the  vertices of $X$ pointwise.

\medskip

\begin{figure}[t]
\centerline{\includegraphics[scale=0.6]{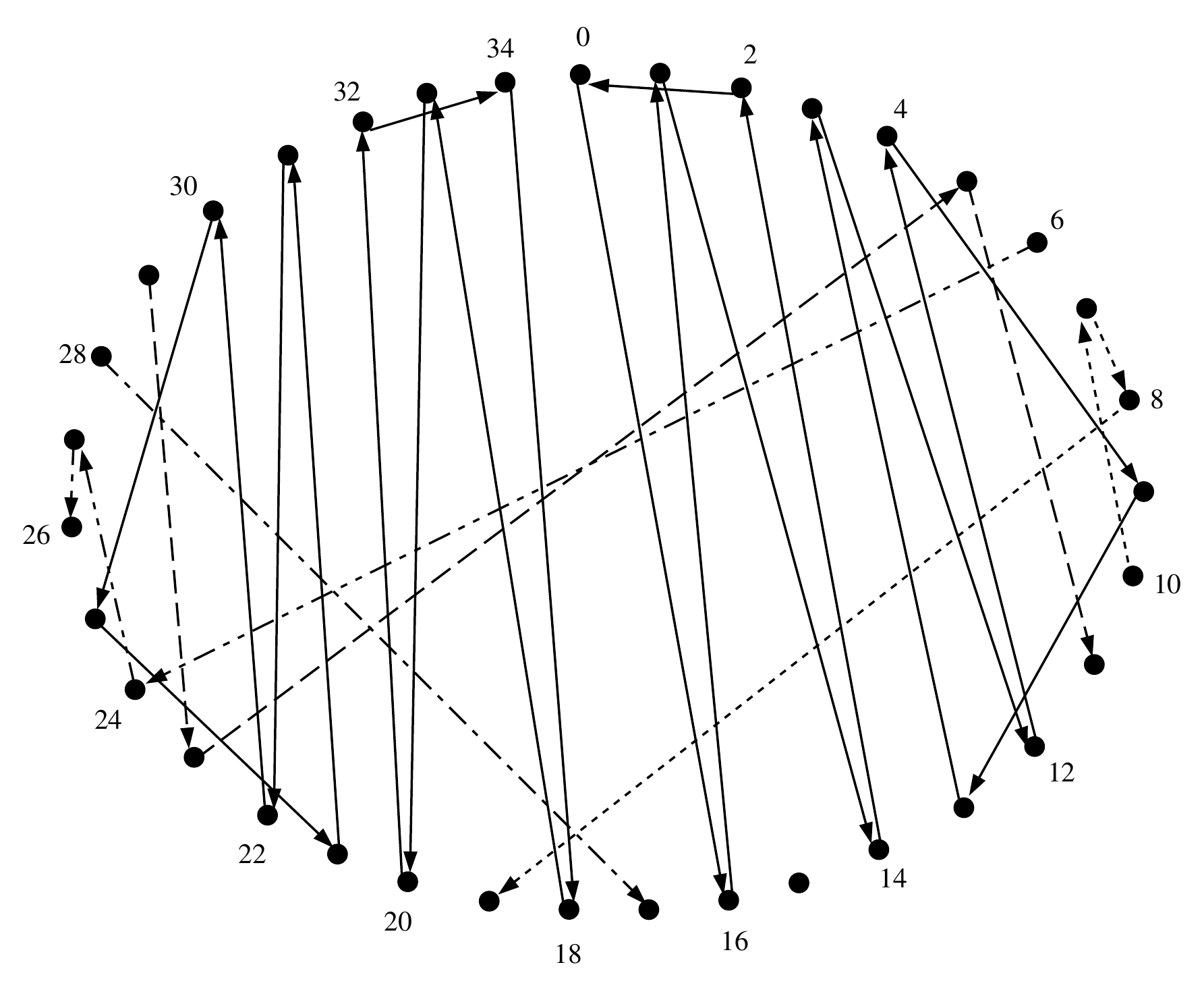}}
\caption{Directed cycles $C_1,\ldots,C_4$ (solid lines) and directed paths $P_1,\ldots,P_4$ (dashed lines) in the construction of a $\vec{C}_5$-factorization of $K_{8[5]}^\ast$. (All the vertices are in $V$, and only their subscripts are specified.)}
\label{fig:n=8(5)}
\end{figure}

{\sc Case 1:} $t=5$. Then $D[V]$ is a circulant digraph with vertex set $V=\{ v_i: i \in \ZZ_{35}\}$ and connection set $\D=\{ \pm d: 1 \le d \le 17, d \not\equiv 0 \pmod{7}\}$.

First, define the following two directed $5$-cycles (see Figure~\ref{fig:n=8(5)}):
\begin{eqnarray*}
C_1 &=& v_0 v_{16} v_1 v_{14} v_2 v_0 \quad \mbox{ and } \\
C_2 &=& v_{13} v_{3} v_{12} v_{4} v_9 v_{13}.
\end{eqnarray*}
The next two directed $5$-cycles are obtained by applying the reflection $\tau: v_i \mapsto v_{-(i+1)}$ to cycles $C_1$ and $C_2$:
\begin{eqnarray*}
C_3 &=& v_{34} v_{18} v_{33} v_{20} v_{32} v_{34} \quad \mbox{ and }  \\
C_4 &=& v_{21} v_{31} v_{22} v_{30} v_{25} v_{21}.
\end{eqnarray*}
Next, we define three directed 3-paths and one directed 1-path:
\begin{eqnarray*}
P_1 &=& v_{6} v_{24} v_{27} v_{26}, \\
P_2 &=& v_{29} v_{23} v_{5} v_{11}, \\
P_3 &=& v_{10} v_{7} v_{8} v_{19}, \quad \mbox{ and }  \\
P_4 &=& v_{28} v_{17}.
\end{eqnarray*}
Observe that these cycles and paths are pairwise disjoint, and $U=V- \bigcup_{i=1}^4 \big( V(P_i) \cup V(C_i)\big)=\{ v_{15} \}$. Their sets of differences are:
\begin{eqnarray*}
\D(C_1) &=& \{ 16,-15,13,-12,-2\}, \\
\D(C_2) &=& \{ -10,9,-8,5,4 \}, \\
\D(C_3) &=& -\D(C_1), \\
\D(C_4) &=& -\D(C_2), \\
\D(P_1) &=& \{ -17,3,-1 \}, \\
\D(P_2) &=& \{ -6,17,6 \}, \\
\D(P_3) &=& \{ -3,1,11 \}, \quad \mbox{ and } \\
\D(P_4) &=& \{ -11 \}.
\end{eqnarray*}
Thus, these paths and cycles jointly use exactly one arc of each difference in $\D$. We next extend the paths to directed 5-cycles as follows:
\begin{eqnarray*}
C_5 &=& P_1 v_{26} x_0 v_{6}, \\
C_6 &=& P_2 v_{11} x_1 v_{29}, \\
C_7 &=& P_3 v_{19} x_2 v_{10}, \quad \mbox{ and }  \\
C_8 &=& P_4 v_{17} x_3 v_{15} x_4 v_{28}.
\end{eqnarray*}
Let $R= C_1 \cup \ldots \cup C_8$, so $R$ is a $\vec{C}_5$-factor in $D$. It is not difficult to verify that
$\{ \rho^i(R): i \in \ZZ_{35} \}$ is a  $\vec{C}_5$-factorization of $D$.

\FloatBarrier
\medskip

{\sc Case 2:} $t=7$. Now $D[V]$ is a circulant digraph with vertex set $V=\{ v_i: i \in \ZZ_{49}\}$ and connection set $\D=\{ \pm d: 1 \le d \le 24, d \not\equiv 0 \pmod{7}\}$.

First, define the following two directed $7$-cycles:
\begin{eqnarray*}
C_1 &=& v_0 v_{23} v_1 v_{21} v_2 v_{20} v_3 v_0 \quad \mbox{ and }  \\
C_2 &=& v_{19} v_{4} v_{17} v_{5} v_{16} v_{6} v_{15} v_{19}.
\end{eqnarray*}
The next two directed $7$-cycles are obtained by applying the reflection $\tau: v_i \mapsto v_{-(i+1)}$ to cycles $C_1$ and $C_2$:
\begin{eqnarray*}
C_3 &=& v_{48} v_{25} v_{47} v_{27} v_{46} v_{28} v_{45} v_{48} \quad \mbox{ and }  \\
C_4 &=& v_{29} v_{44} v_{31} v_{43} v_{32} v_{42} v_{33} v_{29}.
\end{eqnarray*}
The fifth 7-cycle is
$$C_5 = v_{9} v_{10} v_{26} v_{34} v_{40} v_{35} v_{11} v_{9}.$$
Next, we define one directed 5-path and two directed 1-paths:
\begin{eqnarray*}
P_1 &=& v_{39} v_{38} v_{30} v_{14} v_{8} v_{13}, \\
P_2 &=& v_{37} v_{12},  \quad \mbox{ and } \\
P_3 &=& v_{22} v_{24}.
\end{eqnarray*}
Observe that these cycles and paths are pairwise disjoint, and $U=V- \big( \bigcup_{i=1}^5 V(C_i) \cup \big(\bigcup_{i=1}^3 V(P_i) \big)=\{ v_7,v_{18},v_{36},v_{41} \}$.
Their sets of differences are:
\begin{eqnarray*}
\D(C_1) &=& \{ 23,-22,20,-19,18,-17,-3 \}, \\
\D(C_2) &=& \{ -15,13,-12,11,-10,9,4 \}, \\
\D(C_3) &=& -\D(C_1), \\
\D(C_4) &=& -\D(C_2), \\
\D(C_5) &=& \{ 1,16,8,6,-5,-24,-2 \}, \\
\D(P_1) &=& \{ -1,-8,-16,-6,5 \}, \\
\D(P_2) &=& \{ 24  \}, \quad \mbox{ and } \\
\D(P_3) &=& \{ 2 \}.
\end{eqnarray*}
Thus, these paths and cycles jointly use exactly one arc of each difference in $\D$. We extend paths $P_1,P_2,P_3$ to directed 7-cycles as follows:
\begin{eqnarray*}
C_6 &=& P_1 v_{13}  x_0 v_{39}, \\
C_7 &=& P_2 v_{12} x_1 v_7 x_2 v_{18} x_3 v_{37}, \quad \mbox{ and } \\
C_8 &=& P_3 v_{24} x_4 v_{36} x_5 v_{41} x_6 v_{22}.
\end{eqnarray*}

Let $R= C_1 \cup \ldots \cup C_8$, so $R$ is a $\vec{C}_7$-factor in $D$. It is not difficult to verify that
$\{ \rho^i(R): i \in \ZZ_{49} \}$ is a  $\vec{C}_7$-factorization of $D$.

\medskip

{\sc Case 3:} $t=6k+5$ for an integer $k\ge 1$. Now $D[V]$ is a circulant digraph with vertex set $V=\{ v_i: i \in \ZZ_{42k+35}\}$ and connection set $\D=\{ \pm d: 1 \le d \le 21k+17, d \not\equiv 0 \pmod{7}\}$.

\medskip

\begin{figure}[t]
\centerline{\includegraphics[scale=0.6]{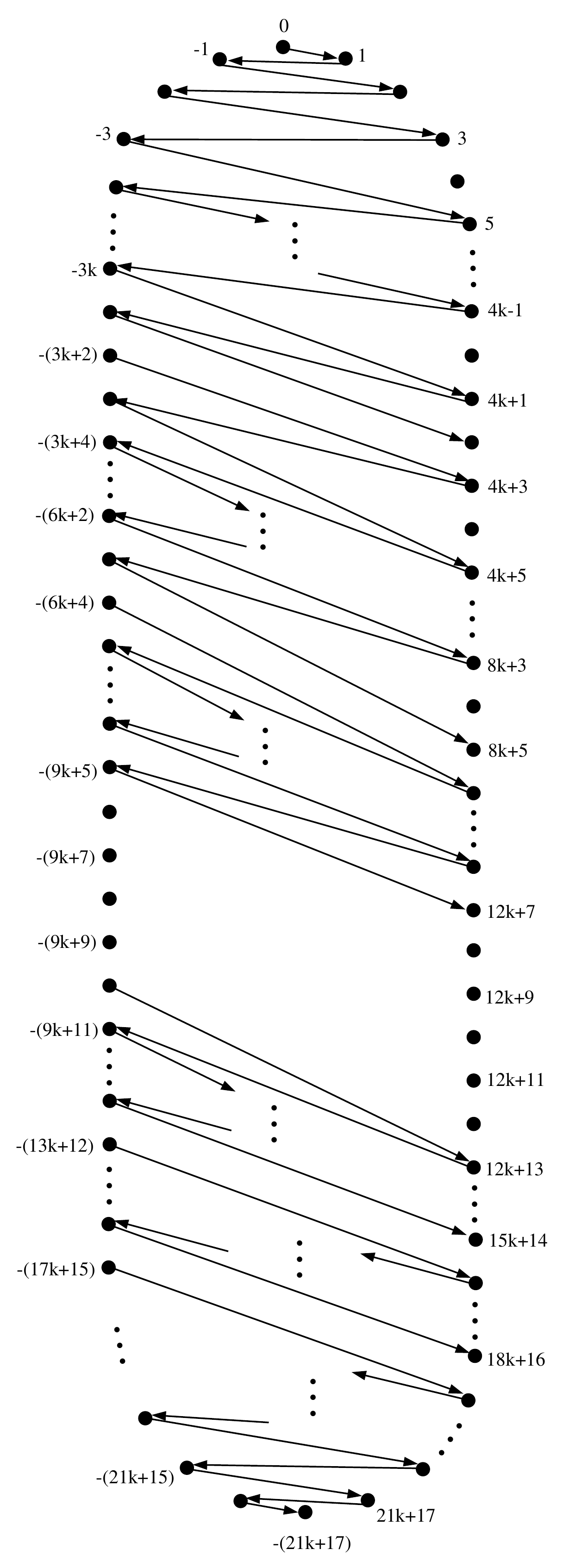}}
\caption{Directed paths $P_1,\ldots,P_6$ in the construction of a $\vec{C}_t$-factorization of $K_{8[t]}^\ast$, case $t=6k+5$, $k \equiv 1$ or $2 \pmod{4}$. (All the vertices are in $V$, and only their subscripts are specified.) }
\label{fig:Fig3}
\end{figure}

{\sc Subcase 3.1:} $k \equiv 1$ or $2 \pmod{4}$.

Define the following three directed $(6k+3)$-paths (see Figure~\ref{fig:Fig3}):
\begin{eqnarray*}
P_1 &=& v_0 v_1 v_{-1} v_2 v_{-2} v_3 v_{-3} v_5 v_{-4} \ldots v_{4k-2} v_{-(3k-1)} v_{4k-1} v_{-3k} v_{4k+1} v_{-(3k+1)} v_{4k+2}, \\
P_2 &=& v_{-(3k+2)} v_{4k+3} v_{-(3k+3)} v_{4k+5} v_{-(3k+4)} v_{4k+6} \ldots v_{-(6k+2)} v_{8k+3} v_{-(6k+3)} v_{8k+5}, \quad \mbox{ and } \\
P_3 &=& v_{-(6k+4)} v_{8k+6} v_{-(6k+5)} v_{8k+7} v_{-(6k+6)} v_{8k+9} \ldots v_{-(9k+4)} v_{12k+6} v_{-(9k+5)} v_{12k+7}.
\end{eqnarray*}
For $i=1,2,3$, let $P_{i+3}$ be the directed $(6k+3)$-path obtained from $P_i$ by applying $\rho^{21k+18}=\rho^{-(21k+17)}$ and changing the direction. Thus,
\begin{eqnarray*}
P_4 &=& v_{-(17k+15)} v_{18k+17}  \ldots  v_{-(21k+16)} v_{-(21k+17)}, \\
P_5 &=& v_{-(13k+12)} v_{15k+15}  \ldots  v_{-(17k+14)} v_{18k+16}, \quad \mbox{ and }  \\
P_6 &=& v_{-(9k+10)} v_{12k+13}  \ldots   v_{-(13k+11)} v_{15k+14} .
\end{eqnarray*}
Observe that these paths are pairwise disjoint, and use all vertices in $V$ except those in
\begin{eqnarray*}
V- \bigcup_{i=1}^6 V(P_i) &=& \{v_4, v_8, v_{12}, \ldots, v_{12k+4} \} \cup \{ v_{12k+8},v_{12k+9},\ldots,v_{12k+12} \} \\
&& \cup \{ v_{-(21k+13)}, v_{-(21k+9)},v_{-(21k+5)}, \ldots, v_{-(9k+13)} \} \\
&&  \cup \{v_{-(9k+9)},v_{-(9k+8)},v_{-(9k+7)},v_{-(9k+6)} \}.
\end{eqnarray*}

The sets of differences of these paths, listing the differences in their order of appearance, are:
\begin{eqnarray*}
\D(P_1) &=& \{ 1,-2,3,-4,5,-6,8,-9, \ldots, -(7k-1),7k+1,-(7k+2),7k+3 \}, \\
\D(P_2) &=& \{ 7k+5,-(7k+6),7k+8,-(7k+9),\ldots, 14k+5,-(14k+6),14k+8 \}, \\
\D(P_3) &=& \{ 14k+10, -(14k+11),14k+12,-(14k+13),14k+15,\ldots \\ && \ldots, 21k+10,-(21k+11),21k+12\}, \\
\D(P_4) &=& -\D(P_1), \\
\D(P_5) &=& -\D(P_2), \quad \mbox{ and }  \\
\D(P_6) &=& -\D(P_3).
\end{eqnarray*}
Thus, these paths jointly use exactly one arc of each difference in $\D-\D'$, where
$$\D'=\{\pm(7k+4),\pm(14k+9),\pm(21k+13),\pm(21k+15),\pm(21k+16),\pm(21k+17)\}.$$
The remaining two directed paths depend on the congruency class of $k$ modulo 4.

\FloatBarrier

\begin{figure}[h]
\centerline{\includegraphics[scale=0.6]{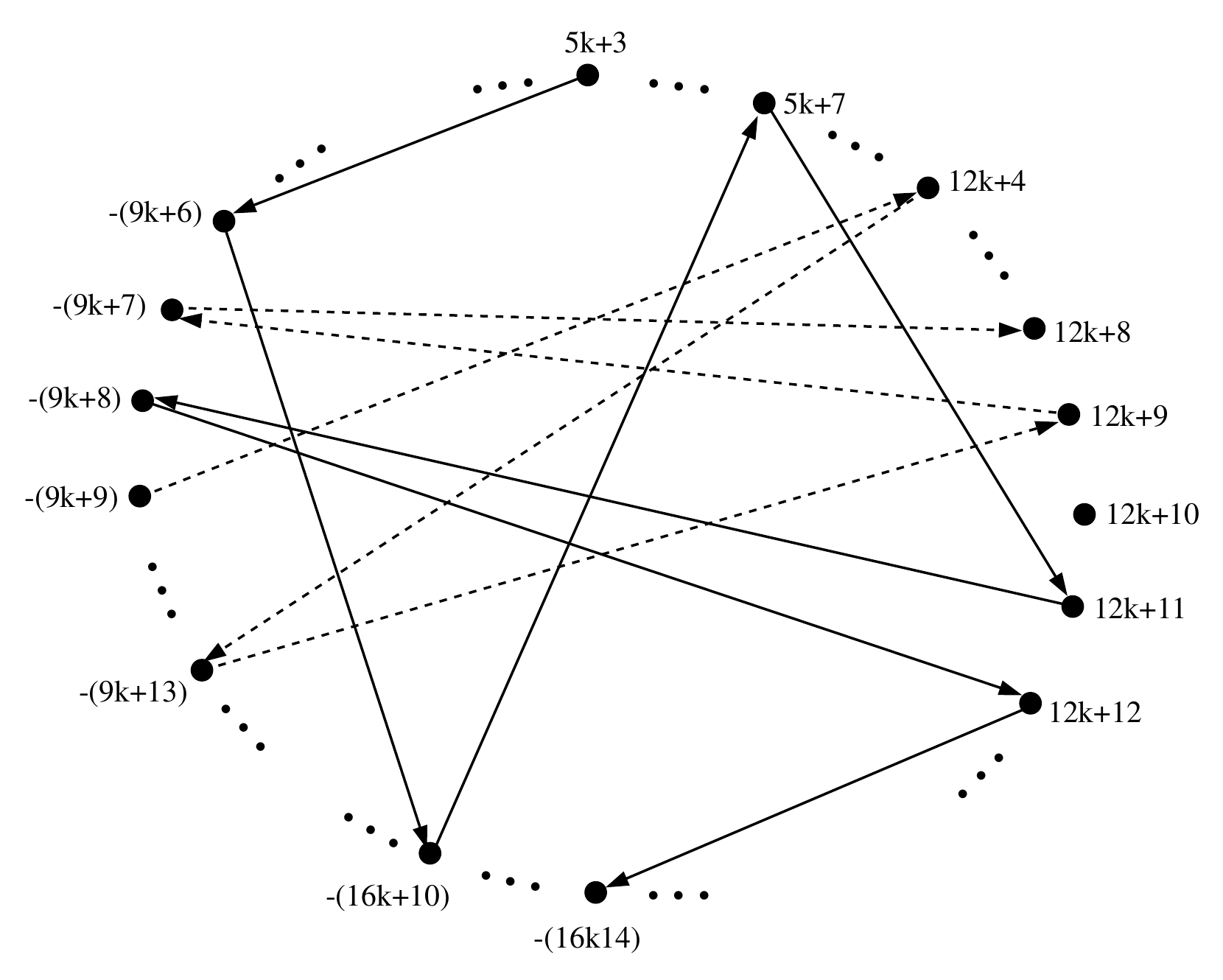}}
\caption{Directed paths $P_7$ and $P_8$ in the construction of a $\vec{C}_t$-factorization of $K_{8[t]}^\ast$, case $t=6k+5$, $k \equiv 1 \pmod{4}$. (All the vertices are in $V$, and only their subscripts are specified.) }
\label{fig:Fig4}
\end{figure}

%\FloatBarrier

If $k\equiv 1 \pmod 4$, we let
\begin{eqnarray*}
P_7 &=& v_{5k+3}, v_{-(9k+6)}, v_{-(16k+10)}, v_{5k+7}, v_{12k+11}, v_{-(9k+8)}, v_{12k+12}, v_{-(16k+14)} \quad \mbox{ and }   \\
P_8 &=&  v_{-(9k+9)}, v_{12k+4}, v_{-(9k+13)}, v_{12k+9}, v_{-(9k+7)}, v_{12k+8}.
\end{eqnarray*}
See Figure~\ref{fig:Fig4}. The sets of differences of these paths are
\begin{eqnarray*}
\D(P_7) &=& \{ -(14k+9),-(7k+4),21k+17,7k+4,21k+16,-(21k+15),14k+9 \} \quad \mbox{ and }   \\
\D(P_8) &=&  \{ 21k+13,-(21k+17),-(21k+13),-(21k+16),21k+15\}.
\end{eqnarray*}

\FloatBarrier

If $k\equiv 2 \pmod 4$, we let
\begin{eqnarray*}
P_7 &=& v_{5k-2}, v_{-(16k+15)}, v_{12k+11}, v_{-(9k+8)}, v_{12k+12}, v_{-(9k+6)},v_{12k+9},v_{-(9k+13)} \quad \mbox{ and }  \\
P_8 &=& v_{-(9k+9)}, v_{12k+10},v_{5k+6},v_{-(16k+11)},v_{-(9k+7)},v_{5k+2}.
\end{eqnarray*}
In this case, we have
\begin{eqnarray*}
\D(P_7) &=& \{ -(21k+13),-(14k+9), 21k+16,-(21k+15),21k+17,21k+15,21k+13 \} \quad \mbox{ and }   \\
\D(P_8) &=&  \{ -(21k+16),-(7k+4),-(21k+17),7k+4, 14k+9 \}.
\end{eqnarray*}
In either case, paths $P_1,\ldots,P_8$ are pairwise disjoint, and jointly contain exactly one arc of each difference in $\D$. Moreover, the set of unused vertices $U$ has cardinality $|U|= (42k+35)-6(6k+4)-8-6=6k-3$. Hence we may label $U=\{ u_i: i \in \ZZ_{6k-3} \}$.

Finally, we extend the eight directed paths to directed $(6k+5)$-cycles as follows. It will be convenient to denote the source and terminal vertex of directed path $P_i$ by $s_i$ and $t_i$, respectively. Let
$$C_i= P_i t_i x_{i-1} s_i \quad \mbox{ for } i=1,2,\ldots,6,$$
while
\begin{eqnarray*}
C_7 &=& P_7 t_7 x_6 u_0 x_7 u_1 \ldots u_{3k-3} x_{3k+4} s_7 \quad \mbox{ and } \\
C_8 &=& P_8 t_8 x_{3k+5} u_{3k-2}  x_{3k+6} u_{3k-1}  \ldots u_{6k-4} x_{6k+4} s_8.
\end{eqnarray*}
To conclude, let $R= C_1 \cup \ldots \cup C_8$, so $R$ is a $\vec{C}_{t}$-factor in $D$. Since the permutation  $\rho$ fixes the  vertices of $X$ pointwise, it is not difficult to verify that
$\{ \rho^i(R): i \in \ZZ_{7t} \}$ is a  $\vec{C}_{t}$-factorization of $D$.

\medskip

{\sc Subcase 3.2:} $k \equiv 0$ or $3 \pmod{4}$. This case will be solved similarly to Subcase 3.1, so we only highlight the differences.

Define the following three directed $(6k+3)$-paths:
\begin{eqnarray*}
P_1 &=& v_0 v_{-1} v_1 v_{-2} v_3 v_{-3} v_5 v_{-4} v_6 v_{-5} v_7 v_{-6} v_9 v_{-7} v_{10} \ldots v_{4k-1} v_{-3k} v_{4k+1} v_{-(3k+1)} v_{4k+2} v_{-(3k+2)}, \\
P_2 &=& v_{4k+3} v_{-(3k+3)} v_{4k+5} v_{-(3k+4)} v_{4k+6} \ldots v_{-(6k+2)} v_{8k+3}  v_{-(6k+3)} v_{8k+5}  v_{-(6k+4)}, \quad \mbox{ and } \\
P_3 &=& v_{8k+6} v_{-(6k+5)} v_{8k+7} v_{-(6k+6)} v_{8k+9} \ldots v_{12k+3} v_{-(9k+3)} v_{12k+5} v_{-(9k+4)} v_{12k+6} v_{-(9k+5)} v_{12k+7} v_{-(9k+6)}.
\end{eqnarray*}
For $i=1,2,3$, let $P_{i+3}$ be the directed $(6k+3)$-path obtained from $P_i$ by applying $\rho^{21k+18}=\rho^{-(21k+17)}$ and changing the direction. Thus,
\begin{eqnarray*}
P_4 &=& v_{18k+16}  v_{-(17k+15)} \ldots  v_{21k+17} v_{-(21k+17)}, \\
P_5 &=& v_{15k+14} v_{-(13k+12)} \ldots  v_{18k+15} v_{-(17k+14)}, \quad \mbox{ and } \\
P_6 &=& v_{12k+12}  v_{-(9k+10)} \ldots  v_{15k+13} v_{-(13k+11)} .
\end{eqnarray*}
Observe that these paths are pairwise disjoint, and use all vertices in $V$ except those in
\begin{eqnarray*}
V- \bigcup_{i=1}^6 V(P_i) &=& \{v_2,v_4, v_8, v_{12}, \ldots, v_{12k+4} \} \cup \{ v_{12k+8},v_{12k+9},\ldots,v_{12k+11} \} \\
&& \cup \{ v_{-(21k+15)},v_{-(21k+13)}, v_{-(21k+9)},v_{-(21k+5)}, \ldots, v_{-(9k+13)} \} \\
&&  \cup \{v_{-(9k+9)},v_{-(9k+8)},v_{-(9k+7)} \}.
\end{eqnarray*}

The sets of differences of these paths, listing the differences in their order of appearance, are:
\begin{eqnarray*}
\D(P_1) &=& \{ -1,2,-3,5,-6,8,-9, \ldots, -(7k-1),7k+1,-(7k+2),7k+3,-(7k+4) \}, \\
\D(P_2) &=& \{ -(7k+6),7k+8,-(7k+9),\ldots, 14k+5,-(14k+6),14k+8,-(14k+9) \}, \\
\D(P_3) &=& \{ -(14k+11),14k+12,-(14k+13),14k+15,\ldots \\ && \ldots, 21k+10,-(21k+11),21k+12,-(21k+13)\}, \\
\D(P_4) &=& -\D(P_1), \\
\D(P_5) &=& -\D(P_2), \quad \mbox{ and } \\
\D(P_6) &=& -\D(P_3).
\end{eqnarray*}
Thus, these paths  jointly use exactly one arc of each difference in $\D-\D'$, where
$$\D'=\{\pm 4, \pm(7k+5), \pm(14k+10), \pm(21k+15), \pm(21k+16), \pm(21k+17)\}.$$
The remaining two directed paths depend on the congruency class of $k$ modulo 4.

If $k\equiv 3 \pmod 4$, we let
\begin{eqnarray*}
P_7 &=& v_{-(9k+9)} v_{12k+10} v_{-(9k+8)} v_{12k+8} v_{-(9k+7)} v_{-(16k+12)} v_{-(16k+16)}  v_{12k+9} \quad \mbox{ and }  \\
P_8 &=& v_{-(21k+13)} v_2 v_{7k+7} v_{-(14k+10)} v_{-(14k+6)} v_4.
\end{eqnarray*}
The sets of differences of these paths are
\begin{eqnarray*}
\D(P_7) &=& \{ -(21k+16),21k+17,21k+16,-(21k+15),-(7k+5),-4,-(14k+10) \} \quad \mbox{ and }   \\
\D(P_8) &=&  \{ 21k+15,7k+5,-(21k+17),4,14k+10 \}.
\end{eqnarray*}

If $k\equiv 0 \pmod 4$, we let
\begin{eqnarray*}
P_7 &=& v_2 v_{-(21k+13)} v_{7k+12} v_{7k+16} v_{-(21k+9)} v_8 v_4 v_{-(21k+15)} \quad \mbox{ and }  \\
P_8 &=& v_{-(16k+13)} v_{-(9k+8)} v_{12k+11} v_{-(9k+9)} v_{12k+9} v_{5k+4}.
\end{eqnarray*}
In this case, we have
\begin{eqnarray*}
\D(P_7) &=& \{ -(21k+15),-(14k+10),4,14k+10,21k+17,-4,21k+16  \} \quad \mbox{ and }   \\
\D(P_8) &=&  \{ 7k+5,-(21k+16),21k+15,-(21k+17),-(7k+5)  \}.
\end{eqnarray*}
The construction is then completed precisely as in Subcase 3.1.

\medskip

{\sc Case 4:} $t=6k+1$ for an integer $k\ge 2$. This case is similar to Subcase 3.2, so we only highlight the differences.
Now $D[V]$ is a circulant digraph with vertex set $V=\{ v_i: i \in \ZZ_{42k+7}\}$ and connection set $\D=\{ \pm d: 1 \le d \le 21k+3, d \not\equiv 0 \pmod{7}\}$.

Define the following three directed $(6k-1)$-paths:
\begin{eqnarray*}
P_1 &=& v_0 v_{-1} v_1 v_{-2} v_3 v_{-3} v_5 v_{-4} v_6 v_{-5} v_7 v_{-6} v_9 v_{-7} v_{10} \ldots v_{4k-2} v_{-(3k-1)} v_{4k-1} v_{-3k}, \\
P_2 &=& v_{4k+1} v_{-(3k+1)} v_{4k+2} v_{-(3k+2)} v_{4k+3} \ldots v_{-(6k-2)} v_{8k-2}  v_{-(6k-1)} v_{8k-1}  v_{-6k}, \quad \mbox{ and } \\
P_3 &=& v_{-(6k+1)} v_{8k} v_{-(6k+2)} v_{8k+1} v_{-(6k+3)} v_{8k+2} v_{-(6k+4)} v_{8k+4} v_{-(6k+5)} v_{8k+5}  \ldots \\
&& \ldots v_{-(9k-3)} v_{12k-6} v_{-(9k-2)} v_{12k-4} v_{-(9k-1)} v_{12k-3}
v_{-9k} v_{12k-2}.
\end{eqnarray*}
For $i=1,2,3$, let $P_{i+3}$ be the directed $(6k-1)$-path obtained from $P_i$ by applying $\rho^{21k+4}=\rho^{-(21k+3)}$ and changing the direction. Thus,
\begin{eqnarray*}
P_4 &=& v_{18k+4}  v_{-(17k+4)} \ldots  v_{21k+3} v_{-(21k+3)}, \\
P_5 &=& v_{15k+4} v_{-(13k+4)} \ldots  v_{18k+3}  v_{-(17k+2)}, \quad \mbox{ and } \\
P_6 &=& v_{-(9k+5)} v_{12k+4}  \ldots  v_{-(13k+3)} v_{15k+3} .
\end{eqnarray*}
Observe that these six paths are pairwise disjoint, and use all vertices in $V$ except those in
\begin{eqnarray*}
V- \bigcup_{i=1}^6 V(P_i) &=& \{v_2,v_4, v_8, v_{12}, \ldots, v_{8k-4} \} \cup \{ v_{8k+3},v_{8k+7},\ldots,v_{12k-5} \} \\
&& \cup \{ v_{12k-1},v_{12k},\ldots,v_{12k+3} \} \\
&& \cup \{ v_{-(21k+1)},v_{-(21k-1)}, v_{-(21k-5)},v_{-(21k-9)}, \ldots, v_{-(13k+7)} \} \\
&& \cup \{ v_{-13k}, v_{-(13k-4)}, \ldots, v_{-(9k+8)} \} \\
&&  \cup \{ v_{-(9k+4)}, v_{-(9k+3)}, v_{-(9k+2)}, v_{-(9k+1)}  \}.
\end{eqnarray*}

The sets of differences of these paths, listing the differences in their order of appearance, are:
\begin{eqnarray*}
\D(P_1) &=& \{ -1,2,-3,5,-6,8,-9, \ldots, ,7k-4,-(7k-3),7k-2,-(7k-1) \}, \\
\D(P_2) &=& \{ -(7k+2),7k+3,-(7k+4),\ldots, 14k-4,-(14k-3),14k-2,-(14k-1) \}, \\
\D(P_3) &=& \{ 14k+1,-(14k+2),14k+3, \ldots, 21k-4,-(21k-3),21k-2\}, \\
\D(P_4) &=& -\D(P_1), \\
\D(P_5) &=& -\D(P_2), \quad \mbox{ and } \\
\D(P_6) &=& -\D(P_3).
\end{eqnarray*}
Thus, these paths jointly use exactly one arc of each difference in $\D-\D'$, where
$$\D'=\{\pm 4, \pm(7k+1), \pm(21k-1), \pm(21k+1), \pm(21k+2), \pm(21k+3)\}.$$
The remaining two directed paths depend on the congruency class of $k$ modulo 4.

If $k\equiv 0 \pmod 4$, we let
\begin{eqnarray*}
P_7 &=& v_{12k-1} v_{-(9k+3)} v_{12k+1} v_{5k} v_{5k-4} v_{-(16k+3)} v_{-(9k+2)} v_{12k} v_{-(9k+1)} v_{12k+2} v_{-(9k+4)} v_{12k-5} \quad \mbox{ and } \\
P_8 &=& v_4 v_8.
\end{eqnarray*}
The sets of differences of these paths are
\begin{eqnarray*}
\D(P_7) &=& \{  -(21k+2),-(21k+3),-(7k+1),-4,-(21k-1),7k+1,21k+2,-(21k+1),\\
&& \;\; 21k+3,21k+1,21k-1 \} \quad \mbox{ and }  \\
\D(P_8) &=&  \{ 4 \}.
\end{eqnarray*}

If $k\equiv 1 \pmod 4$, we let
\begin{eqnarray*}
P_7 &=& v_{12k-1} v_{-(9k+2)} v_{12k+3} v_{-(9k+1)}  v_{12k} v_{5k-1} v_{5k-5} v_{-(16k+4)}  v_{-(9k+3)} v_{12k+1} v_{-(9k+4)} v_{12k-5} \quad \mbox{ and } \\
P_8 &=& v_4 v_8.
\end{eqnarray*}
The sets of differences of these paths are
\begin{eqnarray*}
\D(P_7) &=& \{  -(21k+1), -(21k+2),21k+3,21k+1,-(7k+1),-4,-(21k-1),7k+1, \\
&& \;\; -(21k+3),21k+2,21k-1 \} \quad \mbox{ and }  \\
\D(P_8) &=&  \{ 4 \}.
\end{eqnarray*}

If $k\equiv 2 \pmod 4$, we let
\begin{eqnarray*}
P_7 &=& v_{-(9k+4)} v_{12k+2} v_{-(9k+3)} v_{12k+1} v_{-(9k+1)} v_{12k} v_{-(9k+8)} v_{12k-5} v_{12k-1} v_{5k-2} v_{5k-6} v_{-(16k+5)} \quad \mbox{ and } \\
P_8 &=& v_{5k+2} v_{12k+3}.
\end{eqnarray*}
The sets of differences of these paths are
\begin{eqnarray*}
\D(P_7) &=& \{  -(21k+1), 21k+2,-(21k+3), -(21k+2),21k+1,21k-1, 21k+3,4, \\
&& \;\;  -(7k+1),-4,-(21k-1)\} \quad \mbox{ and }  \\
\D(P_8) &=&  \{ 7k+1 \}.
\end{eqnarray*}

If $k\equiv 3 \pmod 4$, we let
\begin{eqnarray*}
P_7 &=& v_{12k-1} v_{-(9k+3)} v_{12k+3} v_{-(9k+2)}  v_{12k+2} v_{5k+1} v_{5k-3} v_{-(16k+2)}  v_{-(9k+1)} v_{12k} v_{-(9k+4)} v_{12k-5} \quad \mbox{ and } \\
P_8 &=& v_4 v_8.
\end{eqnarray*}
The sets of differences of these paths are
\begin{eqnarray*}
\D(P_7) &=& \{  -(21k+2), -(21k+1), 21k+2, -(21k+3), -(7k+1), -4, -(21k-1), 7k+1, \\
&& \;\; 21k+1, 21k+3, 21k-1 \} \quad \mbox{ and }  \\
\D(P_8) &=&  \{ 4 \}.
\end{eqnarray*}

The construction is then completed similarly to Subcases 3.1 and 3.2, except that $3k-5$ vertices of $X$ and $3k-6$ vertices of $U=V- \cup_{i=1}^6 V(P_i)$ are used to complete $P_7$ to $C_7$, while
$3k$ vertices of $X$ and $3k-1$ vertices of $U$ are used to complete $P_8$ to $C_8$. Observe that, indeed, $|U|=(42k+7)-6(6k)-12-2=6k-7=(3k-6)+(3k-1)$.
\end{proof}

\begin{cor}\label{cor:n=0mod4}
Assume $m \ge 3$ is odd, $n \equiv 0 \pmod{4}$, $t|mn$, and $\gcd(n,t)=1$.
Then $K_{n[m]}^\ast$ admits a $\vec{C}_t$-factorization.
\end{cor}

\begin{proof}
The assumptions imply that $m=st$ for some odd $s$. By Lemmas~\ref{lem:4} and \ref{lem:8}, respectively, the digraphs $K_{4[t]}^\ast$ and $K_{8[t]}^\ast$ admit $\vec{C}_t$-factorizations. Hence by Corollaries~\ref{cor:reduction}(1) and \ref{cor:blow-digraph}(a), the digraphs $K_{n[t]}^\ast$ and $K_{n[m]}^\ast \cong K_{n[t]}^\ast \wr \bar{K_s}$, respectively,  admit $\vec{C}_t$-factorizations.
\end{proof}

%%% (1,6)

\subsection{Subcase $n \equiv 0 \pmod{6}$}

This section covers the smallest of the cases $n=2p$, for $p$ an odd prime. The construction is similar to the case $n=8$. In principle, this approach could be taken to construct a $\vec{C}_t$-factorizaton of $K_{2p[t]}^\ast$ for any fixed prime $p$, however, for $p \ge 5$, the work involved becomes too tedious.

\begin{lemma}\label{lem:6}
Let $t$ be odd, $t \ge 3$. Then $K_{6[t]}^\ast$ admits a $\vec{C}_t$-factorizaton.
\end{lemma}

\begin{proof}
By Corollary~\ref{cor:blow-digraph}(b), we may assume that $t$ is a prime,
and by Theorem~\ref{thm:Ben}, we may assume $t \ge 5$.

Let the vertex set of $D=K_{6[t]}^\ast$ be $V \cup X$, where $V$ and $X$ are disjoint sets, with $V=\{ v_i: i \in \ZZ_{5t}\}$ and $X=\{ x_i: i \in \ZZ_t\}$. The six parts (holes) of $D$ are $X$ and $V_r=\{ v_{5i+r}: i=0,1,\ldots,t-1 \}$, for $r=0,1,\ldots,4$.
Note that $D[V]$ is a circulant digraph with connection set (set of differences) $\D=\{ d \in \ZZ_{5t}: d \not\equiv 0 \pmod{5}\}$. Define a permutation $\rho=( v_0 \, v_1 \, \ldots \, v_{5t-1})$, which fixes the  vertices of $X$ pointwise.

\medskip

{\sc Case 1:} $t=5$. Now $D[V]$ is a circulant digraph with vertex set $V=\{ v_i: i \in \ZZ_{25}\}$ and connection set $\D=\{ \pm 1, \ldots,\pm 4, \pm 6,\ldots,\pm 9, \pm 11,\pm 12\}$.

First, define the following directed $5$-cycle and directed 3-path:
\begin{eqnarray*}
C_1 &=& v_{24} v_{11} v_2 v_{10} v_3 v_{24} \quad \mbox{ and } \\
P_2 &=& v_{6} v_{7} v_{5} v_{8}.
\end{eqnarray*}
The second directed $5$-cycle and directed 3-path are obtained by applying the reflection $\tau: v_i \mapsto v_{-(i+1)}$ to $C_1$ and $P_2$, respectively:
\begin{eqnarray*}
C_3 &=& v_{0} v_{13} v_{22} v_{14} v_{21} v_{0} \quad \mbox{ and } \\
P_4 &=& v_{18} v_{17} v_{19} v_{16}.
\end{eqnarray*}
Next, we define another directed 3-path and a  directed 1-path:
\begin{eqnarray*}
P_5 &=& v_{9} v_{15} v_{1} v_{20}  \quad \mbox{ and } \\
P_6 &=& v_{23} v_{12}.
\end{eqnarray*}
Observe that these cycles and paths are pairwise disjoint, and $U=V- \big( V(C_1) \cup V(P_2) \cup V(C_3) \cup V(P_4) \cup V(P_5)\cup V(P_6) \big)=\{ v_{4} \}$. Their sets of differences are, in order of appearance:
\begin{eqnarray*}
\D(C_1) &=& \{ 12,-9,8,-7,-4 \}, \\
\D(P_2) &=& \{ 1,-2,3 \}, \\
\D(C_3) &=& -\D(C_1), \\
\D(P_4) &=& -\D(P_2), \\
\D(P_5) &=& \{ 6,11,-6 \},  \quad \mbox{ and } \\
\D(P_6) &=& \{ -11 \}.
\end{eqnarray*}
Thus, these paths and cycles jointly use exactly one arc of each difference in $\D$. We next extend the three directed 3-paths $P_2,P_4,P_5$ to directed 5-cycles $C_2,C_4,C_5$ using a distinct vertex in $\{ x_0,x_1,x_2 \}$, and we extend the directed 1-path $P_6$ to a directed 5-cycle $C_6$ using vertices $x_3,v_4,x_4$.

Let $R= C_1 \cup \ldots \cup C_6$, so $R$ is a $\vec{C}_5$-factor in $D$.  Then
$\{ \rho^i(R): i \in \ZZ_{25} \}$ is a  $\vec{C}_5$-factorization of $D$.

\smallskip

\begin{figure}[t]
\centerline{\includegraphics[scale=0.6]{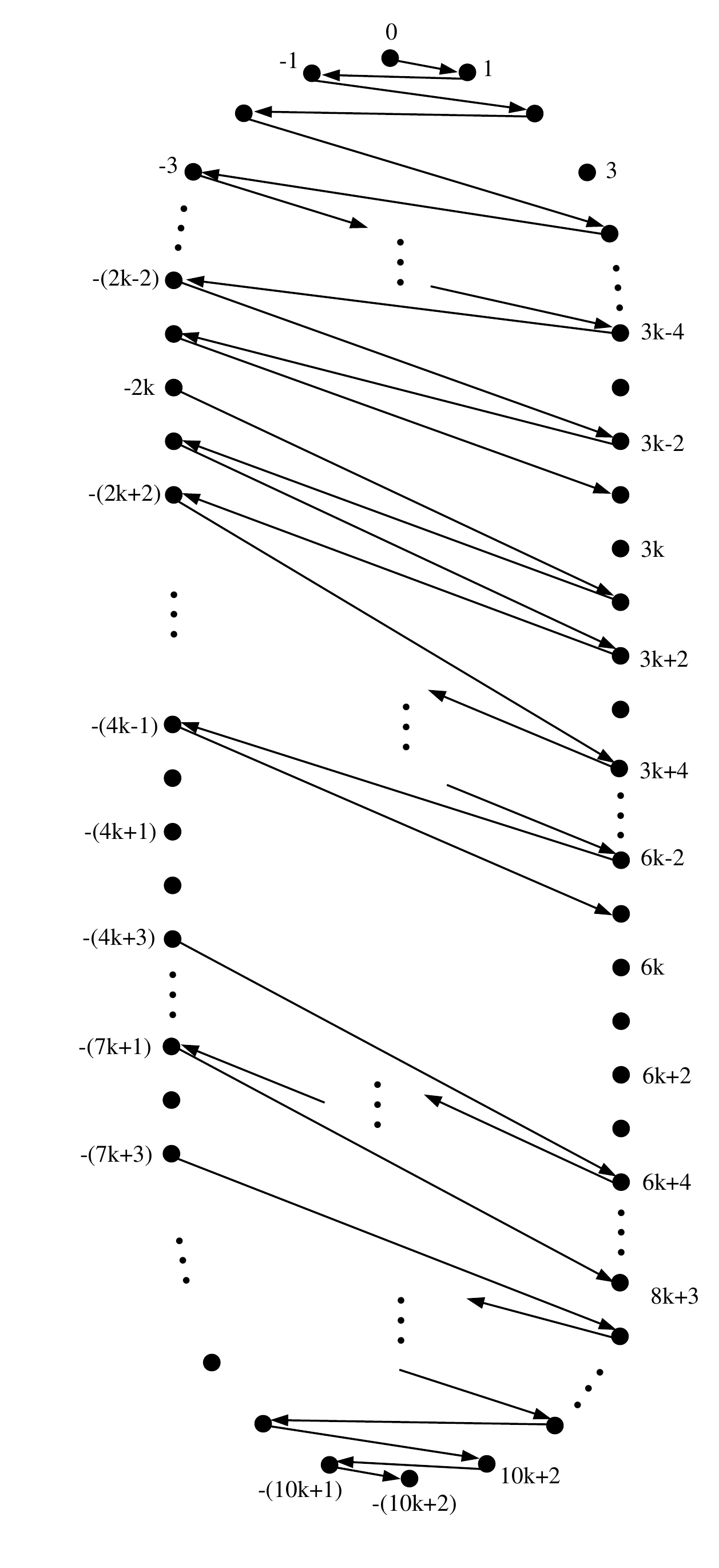}}
\caption{Directed paths $P_1,\ldots,P_4$ in the construction of a $\vec{C}_t$-factorization of $K_{6[t]}^\ast$, case $t=4k+1$. (All the vertices are in $V$, and only their subscripts are specified.) }
\label{fig:Fig5}
\end{figure}

{\sc Case 2:} $t=4k+1$ for an integer $k\ge 2$. Now $D[V]$ is a circulant digraph with vertex set $V=\{ v_i: i \in \ZZ_{20k+5}\}$ and connection set $\D=\{ \pm d: 1 \le d \le 10k+2, d \not\equiv 0 \pmod{5}\}$.

Define the following two directed $(4k-1)$-paths (see Figure~\ref{fig:Fig5}):
\begin{eqnarray*}
P_1 &=& v_0 v_1 v_{-1} v_2 v_{-2} v_4 v_{-3} v_5 v_{-4} v_7 \ldots  v_{-(2k-3)} v_{3k-4} v_{-(2k-2)} v_{3k-2} v_{-(2k-1)} v_{3k-1}  \quad \mbox{ and } \\
P_2 &=& v_{-2k} v_{3k+1} v_{-(2k+1)} v_{3k+2} v_{-(2k+2)} v_{3k+4} \ldots v_{-(4k-3)} v_{6k-4} v_{-(4k-2)} v_{6k-2} v_{-(4k-1)} v_{6k-1}.
\end{eqnarray*}
For $i=1,2$, let $P_{i+2}$ be the directed $(4k-1)$-path obtained from $P_i$ by applying $\rho^{10k+3}=\rho^{-(10k+2)}$ and changing the direction. Thus,
\begin{eqnarray*}
P_3 &=& v_{-(7k+3)} v_{8k+4}  \ldots v_{-(10k+1)} v_{-(10k+2)}  \quad \mbox{ and }\\
P_4 &=& v_{-(4k+3)} v_{6k+4} \ldots v_{-(7k+1)} v_{8k+3}.
\end{eqnarray*}
Observe that these paths are pairwise disjoint, and use all vertices in $V$ except those in
\begin{eqnarray*}
V- \bigcup_{i=1}^4 V(P_i) &=& \{v_3, v_6, v_{9}, \ldots, v_{6k-3} \} \cup \{ v_{6k},v_{6k+1},v_{6k+2},v_{6k+3} \} \\
&& \cup \{ v_{-(10k-1)}, v_{-(10k-4)},v_{-(10k-7)}, \ldots, v_{-(4k+5)} \} \\
&&  \cup \{v_{-(4k+2)},v_{-(4k+1)},v_{-4k} \}.
\end{eqnarray*}

\FloatBarrier

The sets of differences of these paths, listing the differences in their order of appearance, are:
\begin{eqnarray*}
\D(P_1) &=& \{ 1,-2,3,-4,6,-7,8,-9, \ldots, -(5k-6),5k-4,-(5k-3),5k-2 \}, \\
\D(P_2) &=& \{ 5k+1,-(5k+2),5k+3,-(5k+4),5k+6,\ldots, \\
&& \; \ldots, -(10k-6),10k-4,-(10k-3),10k-2 \}, \\
\D(P_3) &=& -\D(P_1),  \quad \mbox{ and }\\
\D(P_4) &=& -\D(P_2).
\end{eqnarray*}
Thus, these paths jointly use exactly one arc of each difference in $\D-\D'$, where
$$\D'=\{\pm(5k-1),\pm(10k-1),\pm(10k+1),\pm(10k+2))\}.$$
The remaining two directed paths depend on the congruency class of $k$ modulo 3.

\begin{figure}[h]
\centerline{\includegraphics[scale=0.6]{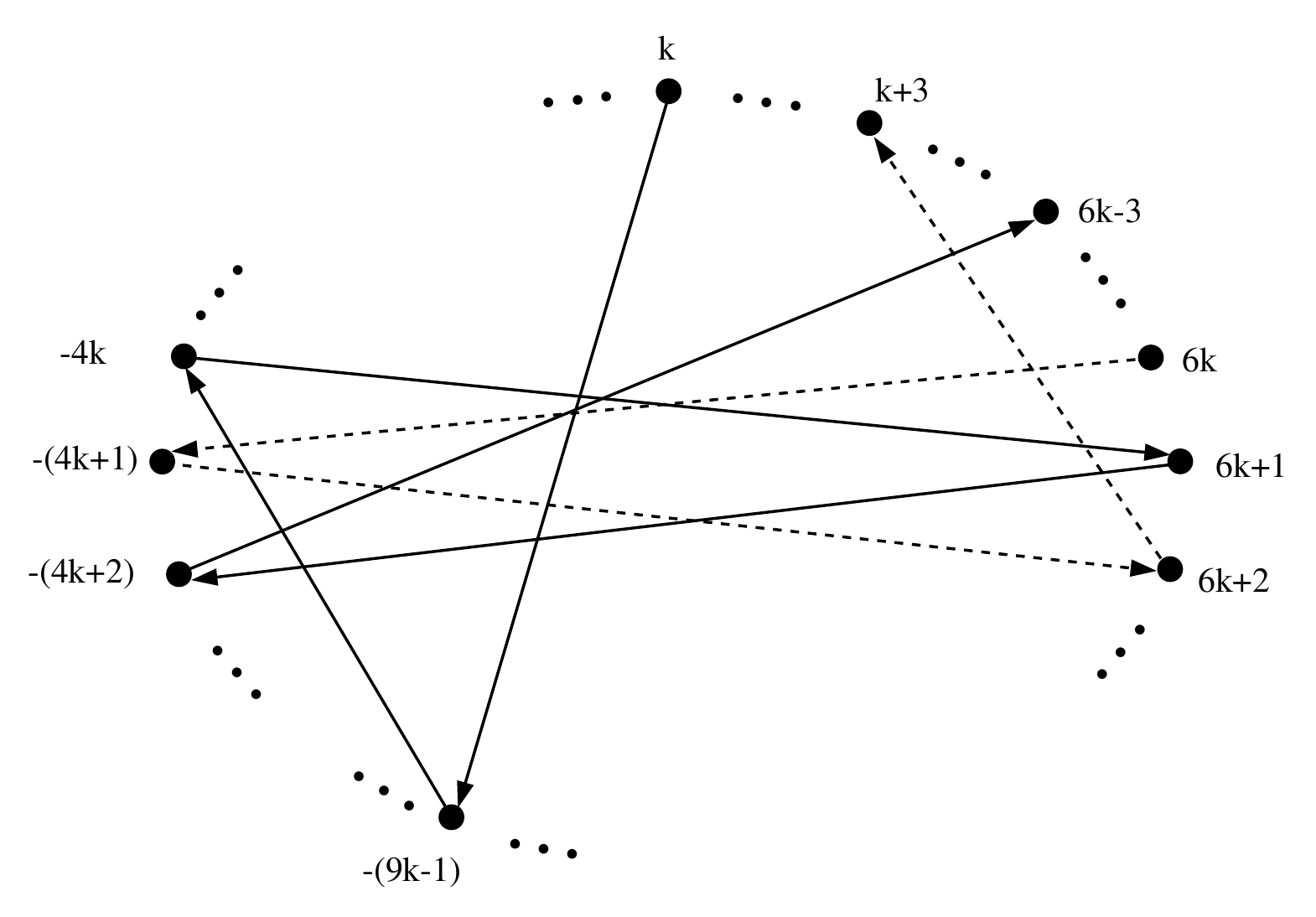}}
\caption{Directed paths $P_5$ and $P_6$ in the construction of a $\vec{C}_t$-factorization of $K_{6[t]}^\ast$, case $t=4k+1$, $k \equiv 0 \pmod{3}$. (All the vertices are in $V$, and only their subscripts are specified.) }
\label{fig:Fig6}
\end{figure}

\FloatBarrier

If $k\equiv 0 \pmod 3$, we let
\begin{eqnarray*}
P_5 &=& v_{k} v_{-(9k-1)} v_{-4k} v_{6k+1} v_{-(4k+2)} v_{6k-3}  \quad \mbox{ and }  \\
P_6 &=& v_{6k} v_{-(4k+1)} v_{6k+2} v_{k+3} .
\end{eqnarray*}
See Figure~\ref{fig:Fig6}. The sets of differences of these paths are
\begin{eqnarray*}
\D(P_5) &=& \{ -(10k-1),5k-1,10k+1,10k+2,10k-1 \}  \quad \mbox{ and }\\
\D(P_6) &=&  \{ -(10k+1),-(10k+2),-(5k-1)\}.
\end{eqnarray*}

If $k\equiv 1 \pmod 3$, we let
\begin{eqnarray*}
P_5 &=& v_{6k+2} v_{-(4k+2)} v_{6k-3} v_{-(4k+5)} v_{6k+1} v_{k+2}  \quad \mbox{ and }\\
P_6 &=& v_{-(4k+1)} v_{6k+3} v_{-4k} v_{k-1}.
\end{eqnarray*}
In this case, we have
\begin{eqnarray*}
\D(P_5) &=& \{ 10k+1, 10k-1, -(10k+2),-(10k-1),-(5k-1) \}  \quad \mbox{ and }\\
\D(P_6) &=&  \{ -(10k+1),10k+2,5k-1\}.
\end{eqnarray*}

If $k\equiv 2 \pmod 3$, we take
\begin{eqnarray*}
P_5 &=& v_{-(4k+2)} v_{6k+2} v_{-(4k+1)} v_{6k} v_{k+1} v_{-(9k-2)}  \quad \mbox{ and } \\
P_6 &=& v_{k+7} v_{-(9k-5)} v_{k+4} v_{6k+3}.
\end{eqnarray*}
The sets of differences are
\begin{eqnarray*}
\D(P_5) &=& \{ -(10k+1), 10k+2, 10k+1,-(5k-1),-(10k-1) \}  \quad \mbox{ and }\\
\D(P_6) &=&  \{ -(10k+2),10k-1,5k-1\}.
\end{eqnarray*}

In all three cases, paths $P_1,\ldots,P_6$ are pairwise disjoint, and jointly contain exactly one arc of each difference in $\D$. Moreover, the set of unused vertices $U$ has cardinality $|U|= (20k+5)-4 \cdot 4k-6-4=4k-5$. Hence we may label $U=\{ u_i: i \in \ZZ_{4k-5} \}$.

Finally, we extend the four directed paths $P_1,\ldots,P_4$ to disjoint directed $(4k+1)$-cycles by adjoining one vertex from $\{ x_0, \ldots, x_3 \}$ to each, extend the directed 5-path $P_5$ to a directed  $(4k+1)$-cycle $C_5$ by adjoining vertices $x_4,u_0,x_5,u_1,\ldots, u_{2k-4}, x_{2k+1}$, and extend the directed 3-path $P_6$ to a directed  $(4k+1)$-cycle $C_6$ by adjoining vertices $x_{2k+2},u_{2k-3},x_{2k+3},u_{2k-2},\ldots,$ $ u_{4k-6}, x_{4k}$.

Finally, let $R= C_1 \cup \ldots \cup C_6$, so $R$ is a $\vec{C}_{t}$-factor in $D$. Since the permutation  $\rho$ fixes the  vertices of $X$ pointwise, it is not difficult to verify that
$\{ \rho^i(R): i \in \ZZ_{5t} \}$ is a  $\vec{C}_{t}$-factorization of $D$.

\smallskip

{\sc Case 3:} $t=4k+3$ for an integer $k\ge 1$. Now $D[V]$ is a circulant digraph with vertex set $V=\{ v_i: i \in \ZZ_{20k+15}\}$ and connection set $\D=\{ \pm d: 1 \le d \le 10k+7, d \not\equiv 0 \pmod{5}\}$.

Define the following two directed $(4k+1)$-paths:
\begin{eqnarray*}
P_1 &=& v_0 v_1 v_{-1} v_2 v_{-2} v_4 v_{-3} v_5 v_{-4} v_7 \ldots  v_{-(2k-2)} v_{3k-2} v_{-(2k-1)} v_{3k-1} v_{-2k} v_{3k+1}  \quad \mbox{ and } \\
P_2 &=& v_{-(2k+1)} v_{3k+2} v_{-(2k+2)} v_{3k+4} \ldots  v_{-(4k-1)} v_{6k-1} v_{-4k} v_{6k+1}v_{-(4k+1)} v_{6k+2}.
\end{eqnarray*}
For $i=1,2$, let $P_{i+2}$ be the directed $(4k+1)$-path obtained from $P_i$ by applying $\rho^{10k+8}=\rho^{-(10k+7)}$ and changing the direction. Thus,
\begin{eqnarray*}
P_3 &=& v_{-(7k+6)} v_{8k+8} \ldots v_{-(10k+6)} v_{-(10k+7)}  \quad \mbox{ and } \\
P_4 &=& v_{-(4k+5)} v_{6k+7} \ldots v_{-(7k+5)} v_{8k+7}.
\end{eqnarray*}
Observe that these paths are pairwise disjoint, and use all vertices in $V$ except those in
\begin{eqnarray*}
V- \bigcup_{i=1}^4 V(P_i) &=& \{v_3, v_6, v_{9}, \ldots, v_{6k} \} \cup \{ v_{6k+3},v_{6k+4},v_{6k+5},v_{6k+6} \} \\
&& \cup \{ v_{-(10k+4)}, v_{-(10k+1)},v_{-(10k-2)}, \ldots, v_{-(4k+7)} \} \\
&&  \cup \{v_{-(4k+4)},v_{-(4k+3)},v_{-(4k+2)} \}.
\end{eqnarray*}

The sets of differences of these paths, listing the differences in their order of appearance, are:
\begin{eqnarray*}
\D(P_1) &=& \{ 1,-2,3,-4,6,-7,8,-9, \ldots, -(5k-3),5k-2,-(5k-1),5k+1 \}, \\
\D(P_2) &=& \{ 5k+3,-(5k+4),5k+6,\ldots,  \ldots, -(10k-1),10k+1,-(10k+2),10k+3 \}, \\
\D(P_3) &=& -\D(P_1),  \quad \mbox{ and }\\
\D(P_4) &=& -\D(P_2).
\end{eqnarray*}
Thus, these paths jointly use exactly one arc of each difference in $\D-\D'$, where
$$\D'=\{\pm(5k+2),\pm(10k+4),\pm(10k+6),\pm(10k+7)\}.$$
The remaining two directed paths depend on the congruency class of $k$ modulo 3. Since $t=4k+3$ is prime, we may assume $k \not\equiv 0 \pmod{3}$.

If $k\equiv 1 \pmod 3$, we let
\begin{eqnarray*}
P_5 &=& v_{6k+6} v_{-(4k+3)} v_{-(9k+5)} v_{k+2} v_{6k+4} v_{-(4k+7)}  \quad \mbox{ and }\\
P_6 &=&  v_{6k} v_{-(4k+4)} v_{6k+5} v_{-(4k+2)}.
\end{eqnarray*}
The sets of differences of these paths are
\begin{eqnarray*}
\D(P_5) &=& \{ 10k+6,-(5k+2),10k+7,5k+2,10k+4 \}  \quad \mbox{ and }\\
\D(P_6) &=&  \{ -(10k+4),-(10k+6),-(10k+7) \}.
\end{eqnarray*}

If $k\equiv 2 \pmod 3$, we take
\begin{eqnarray*}
P_5 &=& v_{-(4k+4)} v_{6k+5} v_{-(4k+3)} v_{6k+3} v_{k+1} v_{-(9k+3)}  \quad \mbox{ and } \\
P_6 &=& v_{k+7} v_{-9k} v_{k+4} v_{6k+6}.
\end{eqnarray*}
The sets of differences are
\begin{eqnarray*}
\D(P_5) &=& \{ -(10k+6),10k+7,10k+6,-(5k+2),-(10k+4) \}  \quad \mbox{ and }\\
\D(P_6) &=&  \{ -(10k+7),10k+4,5k+2 \}.
\end{eqnarray*}

In both cases, paths $P_1,\ldots,P_6$ are pairwise disjoint, and jointly contain exactly one arc of each difference in $\D$. Moreover, the set of unused vertices $U$ has cardinality $|U|= (20k+15)-4 \cdot (4k+2)-6-4=4k-3$. Hence we may label $U=\{ u_i: i \in \ZZ_{4k-3} \}$.

Finally, we extend the four directed paths $P_1,\ldots,P_4$ to disjoint directed $(4k+3)$-cycles by adjoining one vertex from $\{ x_0, \ldots, x_3 \}$ to each, extend the directed 5-path $P_5$ to a directed  $(4k+3)$-cycle $C_5$ by adjoining vertices $x_4,u_0,x_5,u_1,\ldots, u_{2k-3}, x_{2k+2}$, and extend the directed 3-path $P_6$ to a directed  $(4k+3)$-cycle $C_6$ by adjoining vertices $x_{2k+3},u_{2k-2},x_{2k+4},u_{2k-1},\ldots,$ $ u_{4k-4}, x_{4k+2}$.

The construction is then completed as in Case 2.
\end{proof}

\begin{cor}\label{cor:n=0mod6}
Assume $m \ge 3$ is odd, $n \equiv 0 \pmod{6}$, $t|mn$, and $\gcd(t,n)=1$.
Then $K_{n[m]}^\ast$ admits a $\vec{C}_t$-factorization.
\end{cor}

\begin{proof}
If $n \equiv 0 \pmod{4}$, then Corollary~\ref{cor:n=0mod4} yields the desired result. Hence we may assume that $n=6s$ for $s$ odd.
By Lemma~\ref{lem:6}, there exists a $\vec{C}_t$-factorization of $K_{6[t]}^\ast$. Hence by Corollary~\ref{cor:reduction}(2), there exists a $\vec{C}_t$-factorization of $K_{6s[t]}^\ast$.
\end{proof}

%%% case (4,4)

\section{$\vec{C}_t$-factorizaton of $K_{4[m]}^*$ with $m$ odd and  $\gcd(4,t)=4$}

In this section, we settle the first exception from Proposition~\ref{pro:easy}(1).

\begin{lemma}\label{lem:4,4}
Let $p$ be an odd prime. Then $K_{4[p]}^\ast$ admits
\begin{enumerate}[(a)]
\item a $\vec{C}_{4p}$-factorization and
\item a $\vec{C}_{4}$-factorization.
\end{enumerate}
\end{lemma}

\begin{figure}[t]
\centerline{\includegraphics[scale=0.6]{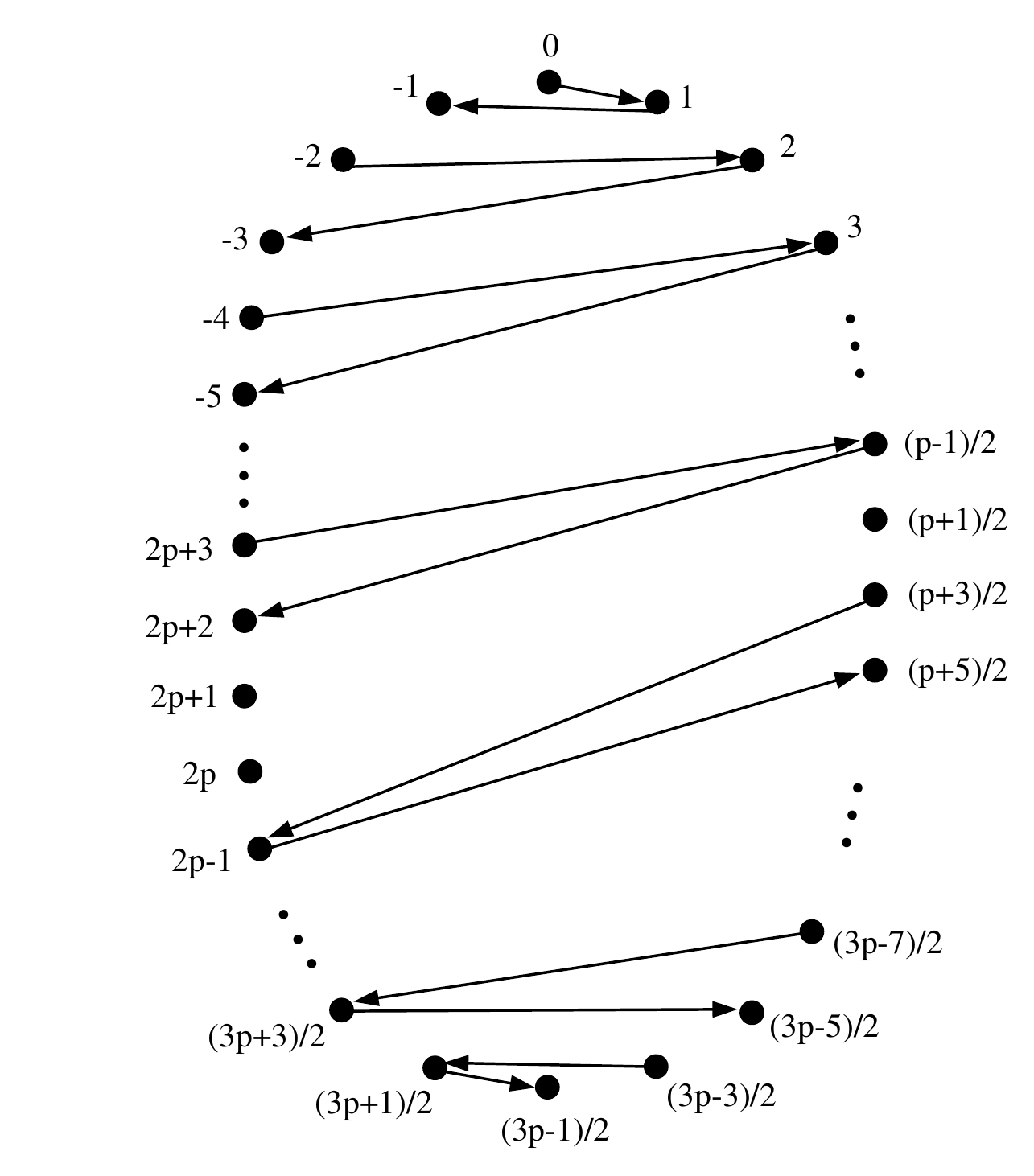}}
\caption{Directed 2-paths $P_0,\ldots,P_{\frac{p-3}{2}},Q_0,\ldots,Q_{\frac{p-3}{2}}$ in the construction of a $\vec{C}_{4p}$-factorization and a $\vec{C}_{4}$-factorization of $K_{4[p]}^\ast$. (All the vertices are in $V$, and only their subscripts are specified.) }
\label{fig:Fig7}
\end{figure}

\begin{proof}
Let the vertex set of $D=K_{4[p]}^\ast$ be $V \cup X$, where $V$ and $X$ are disjoint sets, with $V=\{ v_i: i \in \ZZ_{3p}\}$ and $X=\{ x_i: i \in \ZZ_p\}$. The four parts (holes) of $D$ are $X$ and $V_r=\{ v_{3i+r}: i=0,1,\ldots,p-1 \}$, for $r=0,1,2$.
Note that $D[V]$ is a circulant digraph with connection set (set of differences) $\D=\{ d \in \ZZ_{3p}: d \not\equiv 0 \pmod{3}\}=\{ \pm 1, \pm 2, \pm 4, \pm 5, \pm 7, \ldots,\pm \frac{3p-5}{2},\pm \frac{3p-1}{2} \}$. Let $\rho$ be the permutation $\rho=(v_0 \, v_1 \, \ldots \, v_{3p-1} )$, fixing  $X$ pointwise.

First, for each $i \in \{ 0,1,2,\ldots,\frac{p-3}{2} \}$, we define the directed 2-path
$$P_i= v_{-2i} v_{i+1} v_{-(2i+1)}$$
with the set of differences
$$\D(P_i)=\{ 3i+1,-(3i+2) \}.$$
See Figure~\ref{fig:Fig7}. Let $Q_i$ be the directed 2-path obtained from $P_i$ by applying $\rho^{\frac{3p-1}{2}}$ and reversing the direction; that is,
$$Q_i= v_{-2i+\frac{3p-3}{2}}  v_{i+\frac{3p+1}{2}} v_{-2i+\frac{3p-1}{2}}  $$
and
$$\D(Q_i)=\{ -(3i+1),3i+2 \}.$$
Observe that directed 2-paths $P_0,\ldots,P_{\frac{p-3}{2}},Q_0,\ldots,Q_{\frac{p-3}{2}}$ are pairwise disjoint and use all vertices in the set $V-U$, where
$$U=\left\{ \frac{p+1}{2}, 2p,2p+1 \right\}.$$
Moreover, they jointly use exactly one arc of each difference in
$$\D - \left\{ \pm \frac{3p-1}{2} \right\}.$$
The rest of the construction depends on the statement to be proved.

\FloatBarrier

\begin{enumerate}[(a)]
\item Let
$$P=v_{2p+1} v_{\frac{p+1}{2}}, \quad Q=Q_0 v_{\frac{3p-1}{2}} v_0 P_0, \quad \mbox{ and } \quad R=v_{2p},$$
so $P$ is directed 1-path, $Q$ is a directed 5-path, and $R$ is a directed 0-path,
with
$$\D(P)=\left\{ \frac{3p-1}{2} \right\}, \quad
\D(Q)=\left\{ \pm 1, \pm 2, -\frac{3p-1}{2} \right\},  \quad \mbox{ and } \quad \D(R)=\emptyset.$$
Directed paths $P_1,\ldots,P_{\frac{p-3}{2}},Q_1,\ldots,Q_{\frac{p-3}{2}},P,Q,$ and $R$  are pairwise disjoint, use all vertices in the set $V$,  and  jointly use exactly one arc of each difference in $\D$. As there are $p$ of these paths, we can use the $p$ vertices of $X$ to join them into a directed Hamilton cycle $C$ of $D$; for example, as follows:
\begin{eqnarray*}
C &=& P_1 v_{-3} x_0 v_{-4} P_2 v_{-5} x_1 \ldots  Q_{\frac{p-3}{2}} v_{\frac{p+5}{2}} x_{p-4}  v_{2p+1} P v_{\frac{p+1}{2}} x_{p-3} v_{\frac{3p-3}{2}} Q v_{-1}  x_{p-2} v_{2p} R v_{2p}  x_{p-1} v_{-2}.
\end{eqnarray*}
Then $\{ \rho^i(C): i \in \ZZ_{3p} \}$ is the required $\vec{C}_{4p}$-factorization of $D$.

\item Several cases will need to be considered.

{\sc Case} $p \ge 5$. Note that, since $p$ is prime, we have $p \not\equiv 0 \pmod{3}$.

First, define a directed 4-cycle
$$C_0=v_{-1} v_1 v_{\frac{3p+1}{2}} v_{\frac{3p-3}{2}} v_{-1}$$
with $$\D(C_0) = \left\{ \pm 2, \pm \frac{3p-1}{2} \right\}.$$

{\sc Subcase} $p \equiv 1 \pmod{3}$. We will use directed 2-paths $P_1,\ldots,P_{\frac{p-3}{2}},Q_1,\ldots,Q_{\frac{p-3}{2}}$ defined earlier, except that we replace
$$Q_{\frac{p-1}{3}}=v_{\frac{5p-5}{6}}  v_{\frac{11p+1}{6}}  v_{\frac{5p+1}{6}}$$ with
$$Q_{\frac{p-1}{3}}'=v_{\frac{5p+1}{6}} v_{\frac{5p-5}{6}}  v_{\frac{11p+1}{6}}.$$
In addition, we let
$$R=v_0 v_{2p} v_{2p+1}.$$
Observe that directed 2-paths $P_{\frac{p-1}{3}}$, $Q_{\frac{p-1}{3}}'$, and $R$ jointly use each difference in $\{ \pm 1, \pm p,$ $ \pm (p+1) \}$ exactly once, and the $p-2$ paths $R,P_1,\ldots,P_{\frac{p-3}{2}},Q_1,\ldots,Q_{\frac{p-4}{3}} Q_{\frac{p-1}{3}}' Q_{\frac{p+2}{3}} \ldots Q_{\frac{p-3}{2}}$, together with the 4-cycle $C_0$, jointly use each difference in $\D$ exactly once. In addition, these paths and the cycle are pairwise disjoint, and vertices $v_{\frac{p+1}{2}}$ and $v_{\frac{3p-1}{2}}$ are the only vertices of $V$ that lie in none of them. We now use vertices $x_0,\ldots, x_{p-3}$ to complete the $p-2$ directed 2-paths into directed 4-cycles $C_1,\ldots,C_{p-2}$, and finally define
$$C_{p-1}=v_{\frac{p+1}{2}} x_{p-2} v_{\frac{3p-1}{2}} x_{p-1} v_{\frac{p+1}{2}}.$$
Let $F=C_0 \cup C_1 \cup \ldots \cup C_{p-1}$.
Then $\{ \rho^i(F): i \in \ZZ_{3p} \}$ is the required $\vec{C}_{4}$-factorization of $D$.

\smallskip

{\sc Subcase} $p \equiv 2 \pmod{3}$. This subcase is similar, so we only highlight the differences.
Again, we use directed 2-paths $P_1,\ldots,P_{\frac{p-3}{2}},Q_1,\ldots,Q_{\frac{p-3}{2}}$ defined earlier, except that we first replace
directed 2-paths $P_{\frac{p-3}{2}}$ and $Q_{\frac{p-3}{2}}$ with
$$P_{\frac{p-3}{2}}'= v_{\frac{p-1}{2}} v_{2p+3} v_{\frac{p+1}{2}} \quad \mbox{ and } \quad Q_{\frac{p-3}{2}}'=  v_{2p} v_{\frac{p+5}{2}} v_{2p-1},$$
which cover the same differences, namely, $\pm \frac{3p-5}{2}$ and $\pm \frac{3p-7}{2}$, but use vertices $v_{\frac{p+1}{2}}$ and $v_{2p}$ instead of vertices $v_{2p+2}$ and $v_{\frac{p+3}{2}}$, respectively. We also replace
$$Q_{\frac{p-2}{3}}=v_{\frac{5p-1}{6}}  v_{\frac{11p-1}{6}}  v_{\frac{5p+5}{6}}$$ with
$$Q_{\frac{p-2}{3}}'= v_{\frac{5p+5}{6}} v_{\frac{5p-1}{6}} v_{\frac{11p-1}{6}},$$
and additionally define
$$R=v_0 v_{2p+1} v_{2p+2}.$$
Observe that directed 2-paths $P_{\frac{p-2}{3}}$, $Q_{\frac{p-2}{3}}'$, and $R$ jointly use each difference in $\{ \pm 1, \pm (p-1), \pm p \}$ exactly once, and the $p-2$ paths $R,P_1,\ldots,P_{\frac{p-5}{2}},P_{\frac{p-3}{2}}',Q_1,\ldots,Q_{\frac{p-5}{3}}, Q_{\frac{p-2}{3}}', Q_{\frac{p+1}{3}},$ $\ldots, Q_{\frac{p-5}{2}}, Q_{\frac{p-3}{2}}'$, together with the 4-cycle $C_0=v_{-1} v_1 v_{\frac{3p+1}{2}} v_{\frac{3p-3}{2}} v_{-1}$, jointly use each difference in $\D$ exactly once. Again, these paths and the cycle are pairwise disjoint, this time using all vertices in $V$ except $v_{\frac{p+3}{2}}$ and $v_{\frac{3p-1}{2}}$. The construction is now completed as in the previous case.

{\sc Case} $p=3$. We have $V=\{ v_i: i \in \ZZ_9\}$ and $X=\{ x_0,x_1,x_2 \}$. We construct three starter $\vec{C}_4$-factors, and use $\rho^3$, where $\rho=( v_0 \, v_1 \, \ldots, \, v_8 )$, to generate the rest.

Let $\D_3=\left\{ d_r: d = \pm 1, \pm 2,  \pm 4,  r \in \ZZ_3 \right\}$.
It will be helpful to keep track of the {\em base-3 difference} of each arc $(v_i,v_j)$,  defined as $d_r \in \D_3$ such that $j-i \equiv d \pmod{9}$ and $r \equiv i \pmod{3}$.

Define the following directed 4-cycles in $D$:
\begin{align*}
C_0^0 &= v_0 v_7 v_3 x_0 v_0 && C_1^0 = v_7 v_6 v_5 x_0 v_7 && C_2^0 = v_5 v_0 v_4 x_0 v_5 \\
C_0^1 &= v_4 v_2 v_1 x_1 v_4 && C_1^1 = v_3 v_4 v_8 x_1 v_3 && C_2^1 = v_8 v_1 v_3 x_1 v_8 \\
C_0^2 &= v_5 v_6 v_8 x_2 v_5 && C_1^2 = v_1 v_2 v_0 x_2 v_1 && C_2^2 = v_6 v_2 v_7 x_2 v_6
\end{align*}
Observe that each $F_i= C_i^0 \cup C_i^1 \cup C_i^2$ is a $\vec{C}_4$-factor in $D$, and that  $F_0,F_1,F_2$ jointly contain exactly one arc of each base-3 difference in $\D_3$. Moreover, for each $j,r \in \ZZ_3$, the $\vec{C}_4$-factors $F_0,F_1,F_2$ jointly contain exactly one arc of the form $(x_j,v_i)$ with $i \equiv r \pmod{3}$, and exactly one arc of the form $(v_i,x_j)$ with $i \equiv r \pmod{3}$. Consequently,
$\{ \rho^{3k}(F_i): i,k=0,1,2 \}$ is a $\vec{C}_4$-factorization of $K_{4[3]}^\ast$.
\end{enumerate}
\end{proof}

\begin{cor}\label{cor:n=4}
Assume $m \ge 3$ is odd, $t|4m$, and $\gcd(4,t)=4$.
Then $K_{4[m]}^\ast$ admits a $\vec{C}_t$-factorization.
\end{cor}

\begin{proof}
The assumptions imply that $t=4s$ for some odd $s \ge 1$, and $s|m$.

If $s=1$, let $p$ be any prime factor of $m$. Then by Lemma~\ref{lem:4,4}(b), the digraph $K_{4[p]}^\ast$ admits a $\vec{C}_4$-factorization, and it follows from Corollary~\ref{cor:blow-digraph}(a) that $K_{4[m]}^\ast \cong K_{4[p]}^\ast \wr \bar{K_{\frac{m}{p}}}$ admits a $\vec{C}_4$-factorization.

If $s \ge 3$, let $p$ be any prime factor of $s$. Then by Lemma~\ref{lem:4,4}(a), the digraph $K_{4[p]}^\ast$ admits a $\vec{C}_{4p}$-factorization. It now follows from Corollary~\ref{cor:blow-digraph}(b) that $K_{4[s]}^\ast \cong K_{4[p]}^\ast \wr \bar{K_{\frac{s}{p}}}$ admits a $\vec{C}_{4s}$-factorization. Finally, by  Corollary~\ref{cor:blow-digraph}(a), $K_{4[m]}^\ast$ admits a $\vec{C}_{4s}$-factorization.
\end{proof}

%%% case (6,6)

\section{$\vec{C}_t$-factorizaton of $K_{6[m]}^*$ with $m$ odd and  $\gcd(6,t)=6$}

In this section, we settle the second exception from Proposition~\ref{pro:easy}(1).

\begin{lemma}\label{lem:6,6}
Let $p$ be an odd prime. Then $K_{6[p]}^\ast$ admits
\begin{enumerate}[(a)]
\item a $\vec{C}_{6p}$-factorization and
\item a $\vec{C}_{6}$-factorization.
\end{enumerate}
\end{lemma}

\begin{figure}[t]
\centerline{\includegraphics[scale=0.6]{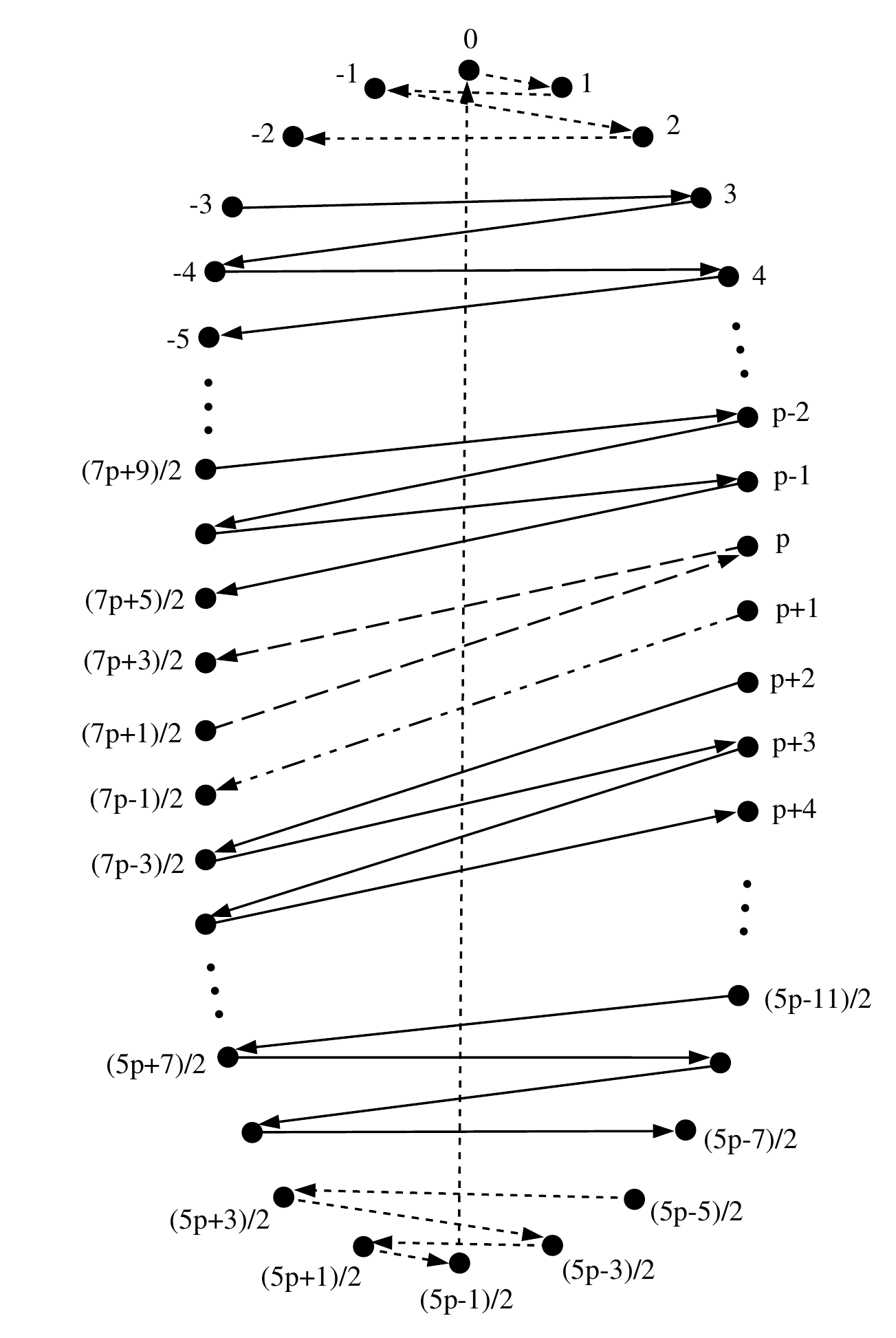}}
\caption{Directed 4-paths $P_1,\ldots,P_{\frac{p-3}{2}},Q_1,\ldots,Q_{\frac{p-3}{2}}$ (solid lines) and directed paths $P_0,R_1,R_2$ (dashed lines) in the construction of a $\vec{C}_{6p}$-factorization of $K_{6[p]}^\ast$. (All the vertices are in $V$, and only their subscripts are specified.) }
\label{fig:Fig8}
\end{figure}

\begin{proof}
Let the vertex set of $D=K_{6[p]}^\ast$ be $V \cup X$, where $V$ and $X$ are disjoint sets, with $V=\{ v_i: i \in \ZZ_{5p}\}$ and $X=\{ x_i: i \in \ZZ_p\}$. The six parts (holes) of $D$ are $X$ and $V_r=\{ v_{5i+r}: i=0,1,\ldots,p-1 \}$, for $r=0,1,\ldots,4$.
Note that  $D[V]$ is a circulant digraph with connection set (set of differences) $\D=\{ d \in \ZZ_{5p}: d \not\equiv 0 \pmod{5}\}=\{ \pm 1, \pm 2, \pm 3,\pm 4, \pm 6,  \ldots, \pm \frac{5p-7}{2}, \pm \frac{5p-3}{2},\pm \frac{5p-1}{2} \}$. Let $\rho$ be the permutation $\rho=(v_0 \, v_1 \, \ldots \, v_{5p-1} )$, which fixes $X$ pointwise.

First, for each $i \in \{ 1,2,\ldots,\frac{p-3}{2} \}$, we define the directed 4-path
$$P_i= v_{-3i} v_{2i+1} v_{-(3i+1)} v_{2i+2} v_{-(3i+2)}$$
with the set of differences
$$\D(P_i)=\{ 5i+1,-(5i+2),5i+3,-(5i+4) \}.$$
See Figure~\ref{fig:Fig8}. The rest of the construction depends on the statement to be proved.

\begin{enumerate}[(a)]
\item  Let $Q_i$ be the directed 4-path obtained from $P_i$ by applying $\rho^{\frac{5p-1}{2}}$ and reversing the direction; that is,
$$Q_i= v_{\frac{5p-5}{2}-3i} v_{\frac{5p+3}{2}+2i}
v_{\frac{5p-3}{2}-3i} v_{\frac{5p+1}{2}+2i} v_{\frac{5p-1}{2}-3i}$$
and
$$\D(Q_i)=\{ -(5i+1),5i+2,-(5i+3),5i+4 \}.$$
See Figure~\ref{fig:Fig8}. Observe that directed 4-paths $P_1,\ldots,P_{\frac{p-3}{2}},Q_1,\ldots,Q_{\frac{p-3}{2}}$ are pairwise disjoint and use all vertices in the set $V-U$, where
$$U=\left\{ v_{-2},v_{-1},v_0,v_1,v_2,v_p, v_{p+1}, v_{\frac{5p-5}{2}},v_{\frac{5p-3}{2}},v_{\frac{5p-1}{2}},v_{\frac{5p+1}{2}},
v_{\frac{5p+3}{2}},
v_{\frac{7p-1}{2}},v_{\frac{7p+1}{2}},v_{\frac{7p+3}{2}} \right\}.$$
Moreover, they jointly use exactly one arc of each difference in $\D-\D'$, where
$$\D' = \left\{ \pm 1,\pm 2, \pm3, \pm 4, \pm \frac{5p-3}{2},\pm \frac{5p-1}{2} \right\}.$$
Additionally, define directed paths
\begin{eqnarray*}
P_0 &=& v_{\frac{5p-5}{2}} v_{\frac{5p+3}{2}} v_{\frac{5p-3}{2}} v_{\frac{5p+1}{2}} v_{\frac{5p-1}{2}} v_0 v_1 v_{-1} v_2 v_{-2}, \\
R_1 &=& v_{p+1} v_{\frac{7p-1}{2}}, \quad \mbox{ and } \\
R_2 &=& v_{\frac{7p+1}{2}} v_p  v_{\frac{7p+3}{2}},
\end{eqnarray*}
and observe that $\D(P_0 \cup R_1 \cup R_2)=\D'$.

The $p$ directed paths $P_1,\ldots,P_{\frac{p-3}{2}},Q_1,\ldots,Q_{\frac{p-3}{2}},P_0,R_1,R_2$ are pairwise disjoint, use all vertices in $V$,  and  jointly use exactly one arc of each difference in $\D$. Hence we can use the $p$ vertices of $X$ to join them into a directed Hamilton cycle $C$ of $D$, similarly to the proof of Lemma~\ref{lem:4,4}(a). Then $\{ \rho^i(C): i \in \ZZ_{5p} \}$ is the required $\vec{C}_{6p}$-factorization of $D$.

\FloatBarrier

\item
In this case, let $Q_i$ be the directed 2-path obtained from $P_i$ by applying $\rho^{\frac{5p-3}{2}}$ and reversing the direction; that is,
$$Q_i= v_{\frac{5p-7}{2}-3i} v_{\frac{5p+1}{2}+2i}
v_{\frac{5p-5}{2}-3i} v_{\frac{5p-1}{2}+2i} v_{\frac{5p-3}{2}-3i}$$
and, again,
$$\D(Q_i)=\{ -(5i+1),5i+2,-(5i+3),5i+4 \}.$$
Observe that directed 4-paths $P_1,\ldots,P_{\frac{p-3}{2}},Q_1,\ldots,Q_{\frac{p-3}{2}}$ are pairwise disjoint and use all vertices in the set $V-U$, where
$$U=\left\{ v_{-2},v_{-1},v_0,v_1,v_2,v_p,  v_{\frac{5p-7}{2}},v_{\frac{5p-5}{2}},v_{\frac{5p-3}{2}},v_{\frac{5p-1}{2}},
v_{\frac{5p+1}{2}},
v_{\frac{7p-3}{2}},v_{\frac{7p-1}{2}},v_{\frac{7p+1}{2}},v_{\frac{7p+3}{2}} \right\}.$$
Moreover, they jointly use exactly one arc of each difference in $\D-\D'$, where
$$\D' = \left\{ \pm 1,\pm 2, \pm3, \pm 4, \pm \frac{5p-3}{2},\pm \frac{5p-1}{2} \right\}.$$
Additionally, on vertex set $U$, define the following directed 6-cycle and two directed paths:
\begin{eqnarray*}
C_0 &=& v_{-1} v_2 v_{-2}  v_{\frac{5p-7}{2}} v_{\frac{5p+1}{2}}
v_{\frac{5p-5}{2}} v_{-1}, \\
R_1 &=& v_p v_{\frac{7p-1}{2}} v_{\frac{7p+3}{2}} v_{\frac{7p+1}{2}} v_{\frac{7p-3}{2}}, \quad \mbox{ and }\\
R_2 &=& v_{\frac{5p-1}{2}} v_0 v_1,
\end{eqnarray*}
and observe that $\D(P_0 \cup R_1 \cup R_2)=\D'$.

Observe that the $p-1$ directed paths $P_1,\ldots,P_{\frac{p-3}{2}},Q_1,\ldots, Q_{\frac{p-3}{2}},R_1,R_2$, together with the directed 6-cycle $C_0$, jointly use each difference in $\D$ exactly once. In addition, these paths and the cycle are pairwise disjoint, using each vertex in $V$ except $v_{\frac{5p-3}{2}}$. We now use vertices $x_0,\ldots, x_{p-3}$ to complete the $p-2$ directed 4-paths
$P_1,\ldots,P_{\frac{p-3}{2}},Q_1,\ldots, Q_{\frac{p-3}{2}},R_1$ into directed 6-cycles $C_1,\ldots,C_{p-2}$, and finally use vertices $x_{p-2},v_{\frac{5p-3}{2}},x_{p-1}$ to complete the directed 2-path $R_2$ to a directed 6-cycle $C_{p-1}$.

Let $F=C_0 \cup \ldots \cup C_{p-1}$. Then $\{ \rho^i(F): i \in \ZZ_{5p}\}$ is a $\vec{C}_{6}$-factorization of $D$.
\end{enumerate}
\end{proof}

\begin{cor}\label{cor:n=6}
Assume $m \ge 3$ is odd, $t|6m$, and $\gcd(6,t)=6$.
Then $K_{6[m]}^\ast$ admits a $\vec{C}_t$-factorization.
\end{cor}

\begin{proof}
The proof is analogous to the proof of Corollary~\ref{cor:n=4}, using Lemma~\ref{lem:6,6} instead of Lemma~\ref{lem:4,4}.
\end{proof}

%%% (3,6)

\section{$\vec{C}_t$-factorizaton of $K_{6[m]}^*$ with $m$ odd and  $\gcd(6,t)=3$}

We shall now address the exceptional case from Proposition~\ref{pro:easy}(2).

\begin{lemma}\label{lem:3,6}
Let $p$ be an odd prime. Then $K_{6[p]}^\ast$ admits
\begin{enumerate}[(a)]
\item a $\vec{C}_{3p}$-factorization and
\item if $p \le 37$, also a $\vec{C}_{3}$-factorization.
\end{enumerate}
\end{lemma}

\begin{proof}
As in the proof of Lemma~\ref{lem:6,6}, let the vertex set of $D=K_{6[p]}^\ast$ be $V \cup X$, where  $V=\{ v_i: i \in \ZZ_{5p}\}$ and $X=\{ x_i: i \in \ZZ_p\}$, so $D[V]$ is a circulant digraph with connection set  $\D=\{ d \in \ZZ_{5p}: d \not\equiv 0 \pmod{5}\}$. Let $\rho$ be the permutation $\rho=(v_0 \, v_1 \, \ldots \, v_{5p-1} )$.

\begin{enumerate}[(a)]
\item First assume $p \ge 5$. This is very similar to the proof of Lemma~\ref{lem:6,6}(a).

For each $i \in \{ 1,2,\ldots,\frac{p-3}{2} \}$, define the directed 4-paths $P_i$ and $Q_i$, as well as directed paths $P_0, R_1$, and $R_2$ exactly as in the proof of Lemma~\ref{lem:6,6}(a). (See Figure~\ref{fig:Fig8}.)
Recall that the $p$ directed paths $P_1,\ldots,P_{\frac{p-3}{2}},Q_1,\ldots,Q_{\frac{p-3}{2}},P_0,R_1,R_2$ are pairwise disjoint, use all vertices in $V$,  and  jointly use exactly one arc of each difference in $\D$.

We join the $\frac{p-1}{2}$ directed paths $P_0, R_2, P_1,\ldots,P_{\frac{p-5}{2}}$ into a directed cycle $C_1$ using $\frac{p-1}{2}$ vertices of $X$. The length of $C_1$ is $9 + 2 + 4 \cdot \frac{p-5}{2} + 2 \cdot \frac{p-1}{2}=3p$, as required. We then join the remaining $\frac{p+1}{2}$ directed paths --- namely, $R_1,P_{\frac{p-3}{2}},Q_1,\ldots,Q_{\frac{p-3}{2}}$ --- into a directed cycle $C_2$ using the remaining $\frac{p+1}{2}$ vertices of $X$. The length of $C_2$ is $1 + 4 \cdot \frac{p-1}{2} + 2 \cdot \frac{p+1}{2}=3p$, again as required.

Let $F=C_1 \cup C_2$. Then $\{ \rho^i(F): i \in \ZZ_{5p} \}$ is a $\vec{C}_{3p}$-factorization of $D$.

Now let $p=3$, so $V=\{ v_i: i \in \ZZ_{15} \}$ and $\D=\{ \pm 1, \pm 2, \pm 3, \pm 4, \pm 6, \pm 7 \}$. Define the directed paths
\begin{eqnarray*}
P_1 &=& v_7 v_0 v_1 v_{-1} v_2 v_{-2} v_{-3} v_3, \\
P_2 &=& v_{-5} v_4 v_{-4}, \quad \mbox{ and }\\
P_3  &=& v_6 v_{-7} v_5 v_{-6},
\end{eqnarray*}
and observe that they are pairwise disjoint, and jointly use exactly one arc of each difference in $\D$. Use vertex $x_0$ to extend $P_1$ to a directed 9-cycle $C_1$, and use vertices $x_1$ and $x_2$ to join $P_1$ and $P_2$ into a directed 9-cycle $C_2$. Then $\{ \rho^i(F): i \in \ZZ_{15} \}$, for $F=C_1 \cup C_2$, is a $\vec{C}_{9}$-factorization of $D$.

\item It suffices to show that $K_{5[p]}^\ast$ admits a spanning subdigraph $D'$ with the following properties:
    \begin{enumerate}[(i)]
    \item $D'$ is a disjoint union of $p$ copies of $\vec{C}_3$ and $p$ copies of $\vec{P}_1$, the directed 1-path; and
    \item $D'$ contains exactly one arc of each difference in $\D$.
    \end{enumerate}
    Let $F$ be obtained from $D'$ by completing each copy of $\vec{P}_1$ to a $\vec{C}_3$ using a distinct vertex in $X$. It then follows that $\{ \rho^i(F): i \in \ZZ_{5p} \}$ is a $\vec{C}_{3}$-factorization of $D$.

    Computationally, we have verified the existence of a suitable subdigraph $D'$ of $K_{5[p]}^\ast$ for all primes $p$, $3 \le p \le 37$ (see Appendix~\ref{app}). Since the existence of a $\vec{C}_{3}$-factorization of $K_{6[p]}^\ast$ for each odd prime $p<17$ is guaranteed by Theorem~\ref{thm:Ben}, only the cases with $17 \le p \le 37$ are presented.
\end{enumerate}
\end{proof}

\begin{cor}\label{cor:g=3}
Assume $m \ge 3$ is odd, $t|6m$, and $\gcd(6,t)=3$.
Then $K_{6[m]}^\ast$ admits a $\vec{C}_t$-factorization, except possibly when $t=3$ and $m$ is not divisible by any prime $p \le 37$.
\end{cor}

\begin{proof}
The assumptions imply that $t=3s$ for some odd $s \ge 1$, and $s|m$.

If $s \ge 3$, let $p$ be any prime factor of $s$. Then by Lemma~\ref{lem:3,6}(a), the digraph $K_{6[p]}^\ast$ admits a $\vec{C}_{3p}$-factorization. It now follows from Corollary~\ref{cor:blow-digraph}(b) that $K_{6[s]}^\ast \cong K_{6[p]}^\ast \wr \bar{K_{\frac{s}{p}}}$ admits a $\vec{C}_{3s}$-factorization. Finally, by  Corollary~\ref{cor:blow-digraph}(a), $K_{6[m]}^\ast$ admits a $\vec{C}_{3s}$-factorization.

If $s=1$, assume $m$ has a prime factor $p \le 37$. Then by Lemma~\ref{lem:3,6}(b), the digraph $K_{6[p]}^\ast$ admits a $\vec{C}_3$-factorization, and it follows from Corollary~\ref{cor:blow-digraph}(a) that $K_{6[m]}^\ast \cong K_{6[p]}^\ast \wr \bar{K_{\frac{m}{p}}}$ admits a $\vec{C}_3$-factorization.
\end{proof}

%%%% main

\section{Proof of Theorem~\ref{thm:main} and conclusion}

For convenience, we re-state the main result of this paper before summarizing its proof.

\begin{customthm}{\ref{thm:main}}
Let $m$, $n$, and $t$ be integers greater than 1, and let $g=\gcd(n,t)$. Assume one of the following conditions holds.
\begin{enumerate}[(i)]
\item $m(n-1)$ even; or
\item $g \not\in \{ 1,3 \}$; or
\item $g=1$, and $n \equiv 0 \pmod{4}$ or $n \equiv 0 \pmod{6}$; or
\item $g=3$, and if $n=6$, then $m$ is divisible by a prime $p \le 37$.
\end{enumerate}
Then the digraph $K_{n[m]}^*$ admits a $\vec{C}_t$-factorization if and only if $t|mn$ and $t$ is even in case $n=2$.
\end{customthm}

\begin{proof}
If $K_{n[m]}^*$ admits a $\vec{C}_t$-factorization, then clearly $t|mn$, and $t$ is even when $n=2$.

Now assume these necessary conditions hold.

If $m(n-1)$ is even, then a $\vec{C}_t$-factorization of $K_{n[m]}^*$ exists by Corollary~\ref{cor:Liu}.

Hence assume $m(n-1)$ is odd.  If $g \not\in \{ 1,3 \}$, then the result follows by Proposition~\ref{pro:easy}, and Corollaries~\ref{cor:n=4} and \ref{cor:n=6}. If $g=1$, the results for $n \equiv 0 \pmod{4}$ and $n \equiv 0 \pmod{6}$ follow by Corollaries~\ref{cor:n=0mod4} and \ref{cor:n=0mod6}, respectively.

Finally, the claim for $g=3$ follows from Proposition~\ref{pro:easy}(2) if $n \ne 6$, and from Corollary~\ref{cor:g=3} if $n=6$ and  $m$ is divisible by a prime $p \le 37$.
\end{proof}

We have thus solved several extensive cases of Problem~\ref{prob}. Since there are no exceptions in the cases with small parameters covered by Theorem~\ref{thm:main}, we propose the following conjecture.

\begin{conj}\label{conj}
Let $m$, $n$, and $t$ be positive integers. Then $K_{n[m]}^*$ admits a $\vec{C}_t$-factorization if and only if $t|mn$, $t$ is even in case $n=2$, and $(m,n,t) \not\in \{ (1,6,3),(1,4,4),(1,6,6) \}$.
\end{conj}

By Corollary~\ref{cor:blow-digraph}, and Lemmas~\ref{lem:filling} and \ref{lem:3,6}, to complete the proof of Conjecture~\ref{conj}, it suffices to prove existence of a $\vec{C}_t$-factorization of $K_{n[m]}^*$ in the following cases:
\begin{enumerate}[(i)]
\item $(m,n,t)=(t,2p,t)$ for a prime $p \ge 5$  and odd prime $t$; and
\item $(m,n,t)=(m,6,3)$ for a  prime $m \ge 41$.
\end{enumerate}

\subsection*{Acknowledgements}

The second author gratefully acknowledges support by the Natural Sciences and Engineering Research Council of Canada (NSERC).

\newpage

\appendix

\section{Starter digraphs for a $\vec{C}_{3}$-factorization of $K_{6[p]}^\ast$}\label{app}

For each prime $p$, $17 \le p \le 37$, we give a set $\C$ containing $p$ copies of $\vec{C}_3$ and a set $\P$ containing $p$ copies of $\vec{P}_1$ that together form a starter digraph $D'$ for a $\vec{C}_{3}$-factorization of $K_{6[p]}^\ast$; see the proof of Lemma~\ref{lem:3,6}(b).

\begin{itemize}
\item $p=17$
\begin{eqnarray*}
\C &=& \{ v_{10} v_{21} v_{33} v_{10}, v_{1} v_{74} v_{63} v_{1},
v_{55} v_{64} v_{83} v_{55}, v_{40} v_{68} v_{49} v_{40},
v_{29} v_{37} v_{58} v_{29}, v_{3} v_{67} v_{59} v_{3}, \\
&&
v_{5} v_{12} v_{36} v_{5}, v_{19} v_{50} v_{26} v_{19},
v_{17} v_{75} v_{23} v_{17}, v_{14} v_{20} v_{47} v_{14},
v_{35} v_{39} v_{71} v_{35}, v_{9} v_{45} v_{13} v_{9},\\
&&
v_{22} v_{70} v_{73} v_{22}, v_{0} v_{82} v_{34} v_{0},
v_{31} v_{77} v_{78} v_{31}, v_{42} v_{81} v_{80} v_{42},
v_{2} v_{44} v_{46} v_{2} \} \\
\P &=& \{ v_{8} v_{6}, v_{65} v_{24}, v_{60} v_{18},
v_{7} v_{79}, v_{56} v_{69},
v_{62} v_{76}, v_{66} v_{52},
v_{41} v_{57}, v_{15} v_{84},
v_{11} v_{28}, v_{4} v_{72}, \\
&&
v_{30} v_{48}, v_{61} v_{43},
v_{16} v_{38}, v_{54} v_{32},
v_{25} v_{51}, v_{53} v_{27} \}
\end{eqnarray*}

\item $p=19$
\begin{eqnarray*}
\C &=& \{ v_{20} v_{32} v_{51} v_{20}, v_{2} v_{78} v_{66} v_{2},
v_{9} v_{26} v_{93} v_{9}, v_{46} v_{74} v_{57} v_{46},
v_{1} v_{10} v_{34} v_{1}, v_{4} v_{37} v_{28} v_{4}, \\
&&
v_{36} v_{70} v_{62} v_{36}, v_{7} v_{33} v_{94} v_{7},
v_{18} v_{77} v_{84} v_{18}, v_{17} v_{83} v_{76} v_{17},
v_{53} v_{59} v_{91} v_{53}, v_{49} v_{87} v_{55} v_{49}, \\
&&
v_{15} v_{69} v_{73} v_{15}, v_{39} v_{80} v_{43} v_{39},
v_{25} v_{67} v_{64} v_{25}, v_{12} v_{65} v_{68} v_{12},
v_{42} v_{86} v_{85} v_{42}, v_{40} v_{41} v_{92} v_{40}, \\
&&
v_{11} v_{58} v_{60} v_{11} \} \\
\P &=& \{ v_{90} v_{88}, v_{14} v_{63}, v_{31} v_{79},
v_{48} v_{61}, v_{21} v_{8},
v_{30} v_{44}, v_{19} v_{5},
v_{3} v_{82}, v_{29} v_{45},
v_{54} v_{72}, v_{89} v_{71}, \\
&&
v_{6} v_{27}, v_{56} v_{35},
v_{0} v_{22}, v_{38} v_{16},
v_{24} v_{47}, v_{75} v_{52},
v_{13} v_{81}, v_{23} v_{50} \}
\end{eqnarray*}

\item $p=23$
\begin{eqnarray*}
\C &=& \{ v_{11} v_{25} v_{44} v_{11}, v_{17} v_{113} v_{99} v_{17},
v_{49} v_{86} v_{73} v_{49}, v_{45} v_{58} v_{82} v_{45},
v_{28} v_{102} v_{114} v_{28}, v_{71} v_{112} v_{83} v_{71}, \\
&&
v_{15} v_{92} v_{103} v_{15}, v_{5} v_{109} v_{32} v_{5},
v_{48} v_{57} v_{91} v_{48}, v_{29} v_{110} v_{101} v_{29},
v_{31} v_{39} v_{75} v_{31}, v_{26} v_{70} v_{34} v_{26}, \\
&&
v_{38} v_{84} v_{77} v_{38}, v_{60} v_{67} v_{106} v_{60},
v_{4} v_{10} v_{52} v_{4}, v_{63} v_{111} v_{69} v_{63},
v_{47} v_{98} v_{94} v_{47}, v_{14} v_{18} v_{65} v_{14}, \\
&&
v_{53} v_{56} v_{105} v_{53}, v_{30} v_{96} v_{93} v_{30},
v_{54} v_{108} v_{107} v_{54}, v_{27} v_{80} v_{81} v_{27},
v_{21} v_{23} v_{79} v_{21} \} \\
\P &=& \{ v_{78} v_{76}, v_{64} v_{8}, v_{37} v_{95},
v_{55} v_{87}, v_{51} v_{19},
v_{72} v_{88}, v_{16} v_{0},
v_{7} v_{24}, v_{85} v_{68},
v_{59} v_{41}, v_{43} v_{61}, \\
&&
v_{3} v_{97}, v_{100} v_{6},
v_{13} v_{35}, v_{42} v_{20},
v_{66} v_{89}, v_{12} v_{104},
v_{36} v_{62}, v_{1} v_{90},
v_{46} v_{74}, v_{50} v_{22},
v_{9} v_{40}, v_{33} v_{2}\}
\end{eqnarray*}

\item $p=29$
\begin{eqnarray*}
\C &=& \{ v_{5} v_{23} v_{46} v_{5}, v_{21} v_{62} v_{39} v_{21},
v_{54} v_{71} v_{100} v_{54}, v_{63} v_{109} v_{80} v_{63},
v_{2} v_{33} v_{131} v_{2}, v_{12} v_{141} v_{43} v_{12}, \\
&&
v_{1} v_{98} v_{112} v_{1}, v_{60} v_{108} v_{74} v_{60},
v_{9} v_{102} v_{115} v_{9}, v_{31} v_{83} v_{44} v_{31},
v_{20} v_{32} v_{76} v_{20}, v_{36} v_{92} v_{48} v_{36}, \\
&&
v_{4} v_{15} v_{57} v_{4}, v_{41} v_{94} v_{52} v_{41},
v_{7} v_{56} v_{143} v_{7}, v_{55} v_{113} v_{64} v_{55},
v_{77} v_{85} v_{136} v_{77}, v_{0} v_{137} v_{51} v_{0}, \\
&&
v_{49} v_{133} v_{140} v_{49}, v_{69} v_{130} v_{123} v_{69},
v_{16} v_{79} v_{73} v_{16}, v_{29} v_{35} v_{117} v_{29},
v_{59} v_{138} v_{142} v_{59}, v_{37} v_{120} v_{116} v_{37}, \\
&&
v_{14} v_{17} v_{81} v_{14}, v_{47} v_{128} v_{125} v_{47},
v_{65} v_{66} v_{134} v_{65}, v_{11} v_{88} v_{87} v_{11},
v_{19} v_{91} v_{93} v_{19} \} \\
\P &=& \{ v_{30} v_{28}, v_{22} v_{96}, v_{78} v_{6},
v_{110} v_{129}, v_{126} v_{107},
v_{3} v_{24}, v_{135} v_{114},
v_{84} v_{106}, v_{97} v_{75},
v_{95} v_{119}, v_{50} v_{26}, \\
&&
v_{8} v_{34}, v_{53} v_{27},
v_{18} v_{45}, v_{67} v_{40},
v_{10} v_{38}, v_{139} v_{111},
v_{90} v_{122}, v_{104} v_{72},
v_{99} v_{132}, v_{103} v_{70},\\
&&
v_{82} v_{118}, v_{61} v_{25},
v_{68} v_{105}, v_{13} v_{121},
v_{86} v_{124}, v_{127} v_{89},
v_{42} v_{144}, v_{58} v_{101}\}
\end{eqnarray*}

\item $p=31$
\begin{eqnarray*}
\C &=& \{ v_{3} v_{22} v_{51} v_{3}, v_{28} v_{154} v_{135} v_{28},
v_{45} v_{63} v_{91} v_{45}, v_{26} v_{72} v_{44} v_{26},
v_{39} v_{56} v_{90} v_{39}, v_{7} v_{58} v_{24} v_{7}, \\
&&
v_{25} v_{41} v_{77} v_{25}, v_{29} v_{148} v_{132} v_{29},
v_{33} v_{47} v_{86} v_{33}, v_{66} v_{119} v_{80} v_{66},
v_{55} v_{68} v_{112} v_{55}, v_{4} v_{146} v_{48} v_{4}, \\
&&
v_{18} v_{79} v_{67} v_{18}, v_{89} v_{138} v_{150} v_{89},
v_{38} v_{49} v_{96} v_{38}, v_{6} v_{114} v_{17} v_{6},
v_{14} v_{106} v_{115} v_{14}, v_{40} v_{103} v_{94} v_{40}, \\
&&
v_{57} v_{65} v_{121} v_{57}, v_{73} v_{137} v_{81} v_{73},
v_{75} v_{82} v_{141} v_{75}, v_{37} v_{133} v_{126} v_{37},
v_{54} v_{60} v_{122} v_{54}, v_{0} v_{149} v_{87} v_{0}, \\
&&
v_{42} v_{46} v_{113} v_{42}, v_{21} v_{92} v_{88} v_{21},
v_{2} v_{5} v_{74} v_{2}, v_{34} v_{120} v_{117} v_{34},
v_{19} v_{100} v_{101} v_{19}, v_{71} v_{153} v_{152} v_{71}, \\
&&
v_{62} v_{64} v_{140} v_{62} \} \\
\P &=& \{ v_{104} v_{102}, v_{85} v_{9}, v_{129} v_{52},
v_{107} v_{128}, v_{13} v_{147},
v_{83} v_{105}, v_{23} v_{1},
v_{95} v_{118}, v_{43} v_{20}, \\
&&
v_{127} v_{151}, v_{134} v_{110},
v_{116} v_{142}, v_{123} v_{97},
v_{84} v_{111}, v_{59} v_{32},
v_{12} v_{136}, v_{99} v_{130},
v_{8} v_{131}, v_{139} v_{16},\\
&&
v_{76} v_{109}, v_{69} v_{36},
v_{27} v_{145}, v_{61} v_{98},
v_{15} v_{53}, v_{108} v_{70},
v_{10} v_{124}, v_{125} v_{11},
v_{31} v_{144}, v_{143} v_{30},\\
&&
v_{50} v_{93}, v_{78} v_{35}\}
\end{eqnarray*}

\item $p=37$
\begin{eqnarray*}
\C &=& \{ v_{130} v_{153} v_{182} v_{130}, v_{125} v_{177} v_{148} v_{125},
v_{78} v_{100} v_{134} v_{78}, v_{29} v_{85} v_{51} v_{29},\\
&&
v_{33} v_{160} v_{181} v_{33}, v_{63} v_{121} v_{84} v_{63},
v_{88} v_{107} v_{151} v_{88}, v_{20} v_{161} v_{142} v_{20},
v_{28} v_{71} v_{89} v_{28}, v_{0} v_{167} v_{43} v_{0},\\
&&
v_{65} v_{82} v_{131} v_{65}, v_{13} v_{149} v_{132} v_{13},
v_{50} v_{168} v_{184} v_{50}, v_{23} v_{90} v_{39} v_{23},
v_{60} v_{114} v_{128} v_{60}, v_{54} v_{122} v_{68} v_{54}, \\
&&
v_{91} v_{104} v_{163} v_{91}, v_{111} v_{183} v_{124} v_{111},
v_{37} v_{146} v_{158} v_{37}, v_{36} v_{157} v_{48} v_{36},\\
&&
v_{52} v_{164} v_{175} v_{52}, v_{86} v_{159} v_{97} v_{86},
v_{49} v_{58} v_{127} v_{49}, v_{66} v_{144} v_{135} v_{66},
v_{14} v_{22} v_{93} v_{14}, v_{44} v_{123} v_{115} v_{44},\\
&&
v_{35} v_{42} v_{116} v_{35}, v_{95} v_{176} v_{169} v_{95},
v_{15} v_{92} v_{98} v_{15}, v_{96} v_{179} v_{173} v_{96},
v_{6} v_{10} v_{109} v_{6}, v_{57} v_{143} v_{61} v_{57},\\
&&
v_{19} v_{103} v_{106} v_{19}, v_{18} v_{105} v_{21} v_{18},
v_{83} v_{171} v_{172} v_{83}, v_{59} v_{156} v_{155} v_{59},
v_{53} v_{55} v_{147} v_{53} \} \\
\P &=& \{ v_{140} v_{138}, v_{26} v_{120}, v_{154} v_{62},
v_{1} v_{162}, v_{77} v_{101},
v_{152} v_{178}, v_{99} v_{73},
v_{3} v_{30}, v_{31} v_{4},
v_{117} v_{145}, v_{108} v_{80}, \\
&&
v_{41} v_{72}, v_{133} v_{102},
v_{17} v_{170}, v_{24} v_{56},
v_{12} v_{45}, v_{38} v_{5},
v_{16} v_{165}, v_{74} v_{110},
v_{32} v_{70}, v_{113} v_{75},\\
&&
v_{79} v_{118}, v_{34} v_{180},
v_{46} v_{87}, v_{81} v_{40},
v_{25} v_{67}, v_{69} v_{27},
v_{2} v_{141}, v_{150} v_{11},
v_{47} v_{94}, v_{166} v_{119},\\
&&
v_{64} v_{112}, v_{174} v_{126},
v_{7} v_{139}, v_{76} v_{129},
v_{9} v_{137}, v_{136} v_{8}\}
\end{eqnarray*}
\end{itemize}

\end{document}